\documentclass{amsart}
\usepackage{amssymb, amsmath, amsthm, graphics, comment, xspace, enumerate}
\newtheorem{thm}{Theorem}[section]

\newtheorem{lem}[thm]{Lemma}
\newtheorem{prop}[thm]{Proposition}
\theoremstyle{definition}
\newtheorem{defn}[thm]{Definition}
\newtheorem{example}[thm]{Example}
\theoremstyle{remark}
\newtheorem{rem}[thm]{Remark}
\numberwithin{equation}{section}

\begin{document}
\title[Abstract degenerate Volterra integro-differential equations...]{Abstract degenerate Volterra integro-differential equations: inverse generator problem}

\author{Marko Kosti\' c}
\address{Faculty of Technical Sciences,
University of Novi Sad,
Trg D. Obradovi\' ca 6, 21125 Novi Sad, Serbia}
\email{marco.s@verat.net}

{\renewcommand{\thefootnote}{} \footnote{
\\ \text{  }  \ \    {\it Key words and phrases.} abstract degenerate Volterra integro-differential equations, abstract degenerate fractional differential equations, inverse generator problem.
\\  \text{  }  \ \ The author is partially supported by grant 174024 of Ministry
of Science and Technological Development, Republic of Serbia.}}

\begin{abstract}
In this paper, we investigate the inverse generator problem for abstract degenerate Volterra integro-differential equations in locally convex spaces. 
The classes of degenerate $(a,k)$-regularized $C$-resolvent families and degenerate mild $(a,k)$-regularized $(C_{1},C_{2})$-existence
and uniqueness 
families play an important role 
in our analysis.
We provide several illustrative applications of obtained theoretical results, primarily to abstract degenerate fractional differential equations with Caputo derivatives.
\end{abstract}
\maketitle

\section{Introduction and Preliminaries}

Up to now, we do not have any relevant reference which treats the inverse generator problem (sometimes also called $A^{-1}$-problem) for abstract Volterra integro-differential equations, degenerate or non-degenerate in time variable. To the best knowledge of the author, the inverse generator problem has not been considered elsewhere even for the abstract degenerate 
differential equations of first order.

The main purpose of this paper is to analyze the inverse generator problem for abstract degenerate Volterra integro-differential equations in sequentially complete locally convex spaces. We basically investigate the problem of generation of exponentially equicontinuous $(a,k)$-regularized $C$-resolvent families and exponentially equicontinuous mild $(a,k)$-regularized $(C_{1},C_{2})$-existence
and uniqueness 
families by the inverses of closed multivalued linear operators but we also clarify some results about generation of degenerate $C$-distribution semigroups and degenerate $C$-distribution cosine functions by the inverses of closed multivalued linear operators. As mentioned in the abstract, we provide several interesting applications to the abstract degenerate fractional differential equations with Caputo derivatives.

Chronologically speaking, R. deLaubenfels proved in 1988 that any injective infinitesimal generator $A$ of a bounded analytic $C_{0}$-semigroup in a Banach space $E$ has the property that the inverse operator 
$A^{-1}$ also generates a bounded analytic $C_{0}$-semigroup of the same angle (\cite{l11}). In this paper, the author asked whether any  injective infinitesimal generator $A$ of a bounded $C_{0}$-semigroup in $E$ has the property that the inverse operator $A^{-1}$ generates a $C_{0}$-semigroup. As we know today, the answer is negative in general: a simple counterexample is given by H. Komatsu already in 1966, who constructed
an injective infinitesimal generator of a contraction semigroup
on the Banach space $c_{0}$, for which the inverse operator is not an
infinitesimal generator of a $C_{0}$-semigroup (\cite{komatsu}, \cite{gomilko}). Concerning the inverse generator problem, it should be noted that 
A. Gomilko, H. Zwart and Y. Tomilov proved  in 2007 
that the answer to R. deLaubenfels's question is negative in the state space 
$l^{p},$ where $1<p<\infty$ and $p\neq 2$ (see \cite{gomilko}), as well as that S. Fackler 
proved in 2016 that the answer to this question is negative in the state space $L^{p}({\mathbb R}),$ where $1<p<\infty$ and $p\neq 2$ (see \cite{fakl}).
We do not yet know whether there exists an injective infinitesimal generator $A$ of a bounded $C_{0}$-semigroup in a Hilbert space $H$
such that $A^{-1}$ does not generate a $C_{0}$-semigroup in $H$ (it is well known that for any injective infinitesimal generator $A$ of a contraction $C_{0}$-semigroup in $H$, the inverse operator $A^{-1}$ likewise generates a contraction $C_{0}$-semigroup in $H$ by the Lumer-Phillips theorem; see the paper \cite{liuponovi} by R. Liu for the fractional analogue of this result). For further information about the inverse generator problem, we refer the reader to the papers \cite{l1111} by R. deLaubenfels, \cite{tanja} by T. Eisner, H. Zwart, \cite{ponovi}-\cite{ponovidecko} by H. Zwart and the  recent survey \cite{gomilko-banach} by A. Gomilko.

We would like to note that the complexity of inverse generator problem lies also in the fact that the use of real or complex representation theorems for the Laplace transform does not take a satisfactory effect. To explain this in more detail, assume that $A$ is injective and generates a bounded $C_{0}$-semigroup in the Banach space $E$ equipped with the norm $\| \cdot \|.$
Then a simple calculation involving the Hille-Yosida theorem yields that, for every $\lambda>0$ and $n\in {\mathbb N},$ we have
\begin{align*}
 \frac{d^{n}}{d\lambda^{n}}\Bigl[ \bigl(\lambda -A^{-1}\bigr)^{-1}\Bigr]
&=(-1)^{n}n! \Bigl[ \lambda^{-1}-\lambda^{-2}\bigl(\lambda^{-1}-A  \bigr)^{-1}\Bigr]^{n+1}
\\&=(-1)^{n}n! \sum_{k=0}^{n+1}(-1)^{k} \binom{n+1}{k}\lambda^{-(n+1+k)}\bigl(\lambda^{-1}-A \bigr)^{-k},
\end{align*}
so that 
$$
\Biggl\| \frac{d^{n}}{d\lambda^{n}}\Bigl[ \bigl(\lambda -A^{-1}\bigr)^{-1}\Bigr]\Biggr\|\leq \frac{Mn!}{\lambda^{n+1}}\cdot 2^{n+1}.
$$
Since the multiplication with number $2^{n+1}$ has appeared above, this estimate is completely useless if one wants to prove that the operator $A^{-1}$ generates an exponentially bounded $r$-times integrated semigroup 
for some real number $r \geq 0$ (see also \cite[Proposition 3.1, Theorem 3.2, Theorem 3.4]{gomilko-banach}). On the other hand, a simple computation shows that
the resolvent of $A^{-1}$ is bounded in norm by $\mbox{Const} \cdot |\lambda |$ on any right half plane $\{ z\in {\mathbb C} : \Re z>a\},$ where $a>0,$ so that the complex characterization theorem for the Laplace transform immediately yields that the operator $A^{-1}$ generates an exponentially bounded $r$ -times integrated semigroup for any real number $r>2.$ In 2007, S. Piskarev and H. Zwart
proved that the operator $A^{-1}$ generates an exponentially bounded once integrated semigroup (\cite{piskarev}), while M. Li, J. Pastor and S. Piskarev improved this result in 2018 by showing that the operator $A^{-1}$ generates a tempered $r$-times integrated semigroup for any real number $r>1/2.$ Moreover, they formulated a corresponding result for tempered fractional resolvent operator families of order $\alpha \in (0,2];$ see \cite{miao-li} for more details.

The organization and main ideas of this paper can be briefly described as follows. After giving some preliminaries about the Mittag-Leffler functions and Wright functions, fractional calculus and the Bessel functions of first kind, we consider
multivalued linear operators
in Subsection \ref{mlos}; Proposition \ref{lav} is the only original contribution of ours given in this subsection. In Subsection \ref{solution}, we remind ourselves of the basic definitions and results about degenerate $(a,k)$-regularized $C$-resolvent families and degenerate mild $(a,k)$-regularized $(C_{1},C_{2})$-existence
and uniqueness 
families. Our main structural results are formulated and proved in Section \ref{mario}. In Proposition \ref{loza} and Proposition \ref{expaq}, we provide the necessary and sufficient conditions for the generation of exponentially equicontinuous mild $(a,k)$-regularized
$C_{1}$-existence families, mild $(a,k)$-regularized
$C_{2}$-uniqueness families and $(a,k)$-regularized $C$-resolvent families by the inverses of multivalued linear operators. For a given multivalued linear operator ${\mathcal A},$ 
it is important to profile some sufficient conditions under which the inverse operator ${\mathcal A}^{-1}$ is a subgenerator of an exponentially equicontinuous, analytic $(a,k)$-regularized $C$-resolvent family of angle $ \alpha \in (0,\pi/2];$ this is done in Proposition \ref{pope} and Proposition \ref{popep}. 
We reconsider some results from \cite{l111} and \cite{miao-li} 
for abstract degenerate Volterra integro-differential equations in Theorem \ref{prostaku} and Theorem \ref{pretis}.
We also analyze here the situation in which we do not assume the existence of $C$-resolvent of a corresponding multivalued linear operator ${\mathcal A}$
on some right half plane. Moreover, we observe that the existence of $C$-resolvent set of ${\mathcal A}$ at infinity does not play any role for the generation of certain classes of $(a,k)$-regularized $C$-resolvent operator families, $C$-(ultra)distribution semigroups and $C$-(ultra)distribution cosine functions by the inverse of a closed multivalued linear operator ${\mathcal A};$  sometimes it is crucial to investigate the behaviour of $C$-resolvent set of ${\mathcal A}$ around zero, only (see Example \ref{distsemig}). In Theorem \ref{ras-sar}, which follows from an application of a much more general Theorem \ref{pretis}, we investigate some special classes of fractional solution operator families generated by the inverse operator ${\mathcal A}^{-1}$; in the obtained representation formula, we essentially use 
the Wright functions. 
In contrast to non-degenerate differential equations, we have found the inverse generator problem much more 
important from the point of view of possible applications (see e.g. Example \ref{jove} and Example \ref{vgy}). Especially, applications to degenerate time-fractional equations with abstract differential operators have been investigated in 
Subsection \ref{srbq}, where the notion of integrated $(g_{\alpha},C)$-regularized resolvent families plays a crucial role. In this subsection, we continue our previous research studies from the paper \cite{filomat}. Besides that, we provide several remarks and useful observations about problems under our consideration.

Unless specified otherwise, by
$X$ we denote a Hausdorff sequentially complete
locally convex space \index{sequentially complete
locally convex space!Hausdorff} over the field of complex numbers, SCLCS for
short. The abbreviation $\circledast$ stands for the fundamental system of seminorms
\index{system of seminorms} which defines the topology of $X;$ if $X$ is a Banach space \index{Banach space} and $A$ is linear operator on $X,$ then the norm \index{norm} of an element $x\in X$ is denoted by $\|x\|.$
Assuming that $Y$ is another SCLCS over the field of complex numbers,
$L(X,Y)$ denotes the space consisting of all continuous linear mappings\index{continuous linear mapping} from $X$ into
$Y;$ $L(X)\equiv L(X,X).$ Let ${\mathcal B}$ be the family of bounded subsets\index{bounded subset} of $X,$ let $\circledast_{Y}$ denote the fundamental system of seminorms
\index{system of seminorms} which defines the topology of $Y,$ and
let $p_{B}(T):=\sup_{x\in B}p(Tx),$ $p\in \circledast_{Y},$ $B\in
{\mathcal B},$ $T\in L(X,Y).$ Then $p_{B}(\cdot)$ is a seminorm\index{seminorm} on
$L(X,Y)$ and the system $(p_{B})_{(p,B)\in \circledast_{Y} \times
{\mathcal B}}$ induces the Hausdorff locally convex topology on
$L(X,Y).$ 
It is well known that the space $L(X,Y)$ is sequentially
complete provided that $X$ is barreled.
By $I$ we denote the identity operator on $X,$ and by $\chi_{\Omega}(\cdot)$ we denote the characteristic function of set $\Omega.$

Set $\Sigma_{\theta}:=\{z\in {\mathbb C} \setminus \{0\} : |\arg(z)|<\theta\}$ ($\theta \in (0,\pi]$). The symbol $\ast$ is used to denote the finite convolution product. Suppose that $\alpha >0$ and $m=\lceil \alpha \rceil .$ Then we define
$g_{\alpha}(t):=t^{\alpha -1}/\Gamma (\alpha),$ $t>0,$ where $\Gamma(\cdot)$ denotes the Euler Gamma function.
The Caputo fractional derivative\index{fractional derivatives!Caputo}
${\mathbf D}_{t}^{\alpha}u(t)$ is defined for those functions $u\in
C^{m-1}([0,\infty) : X)$ for which $g_{m-\alpha} \ast
(u-\sum_{k=0}^{m-1}u_{k}g_{k+1}) \in C^{m}([0,\infty) : X),$
by
$$
{\mathbf
D}_{t}^{\alpha}u(t):=\frac{d^{m}}{dt^{m}}\Biggl[g_{m-\alpha}
\ast \Biggl(u-\sum_{k=0}^{m-1}u_{k}g_{k+1}\Biggl)\Biggr].
$$

The following condition on a scalar valued function $k(t)$ will be used in the sequel:
\begin{itemize}
\item[(P1):]
$k(t)$ is Laplace transformable, i.e., it is locally integrable on\index{function!Laplace transformable}
$[0,\infty)$ and there exists $\beta \in {\mathbb R}$ such that\\
$\tilde{k}(\lambda):={\mathcal L}(k)(\lambda):=\lim_{b
\rightarrow \infty}\int_{0}^{b}e^{-\lambda t}k(t)\, dt:=\int
^{\infty}_{0}e^{-\lambda t}k(t)\, dt$ exists for all $\lambda \in
{\mathbb C}$ with $\Re\lambda>\beta.$ Put $\text{abs}(k):=$inf$\{
\Re\lambda : \tilde{k}(\lambda) \mbox{ exists} \}.$
\end{itemize}

For further information concerning the Laplace transform of functions with values in  SCLCSs, we refer the reader to \cite{FKP} and \cite{x263}. For the Banach space case, see \cite{a43}.

In the continuation, we will use the Bessel functions of first kind. Let us recall that the Bessel function of order $\nu>0$, denoted by $J_{\nu}$, is defined by 
$$
J_{\nu}(z):=\Bigl( \frac{z}{2}\Bigr)^{\nu} \sum_{n=0}^{\infty}\frac{(-1)^{k}(z/2)^{2k}}{k! \Gamma (\nu +k+1)},\quad z\in {\mathbb C}.
$$
Then for each $\nu >0$ we have the existence of a finite real constant $M>0$ such that $\lim_{r\rightarrow +\infty}r^{1/2}J_{\nu}(r)=0.$
The following Laplace transform identity, which holds true for each $\beta \geq 0,$ plays an important role in \cite{miao-li}:
\begin{align}\label{laplace}
\int^{\infty}_{0}e^{-\lambda t}J_{1+\beta}\bigl( 2\sqrt{st} \bigr)s^{(1+\beta)/2}\, ds=t^{(1+\beta)/2}\lambda^{-2-\beta}e^{-t/\lambda},\quad \Re \lambda >0,\ t>0.
\end{align}
It is worth noting that the representation formula obtained in \cite[Theorem 4.2(i)]{miao-li} with the help of \eqref{laplace} is motivated by earlier results of R. deLaubenfels established in \cite[Theorem 3.3, Corollary 3.7, Theorem 4.3(f), Corollary 4.5, Corollary 4.7, Proposition 4.11]{l111}. In all these results, the Laplace transform identities for various Bessel type functions play a crucial role.  

For further information about the Bessel functions, we refer the reader to the monograph of G. N. Watson \cite{watson}. We will also need the following class of Wright functions
$$
\phi(\rho,\nu; z):=\sum_{n=0}^{\infty}\frac{z^{n}}{n! \Gamma (\rho n+\nu)},\quad z\in {\mathbb C}\ \ (\rho>-1,\ \nu \in {\mathbb C}),
$$
which is well known because of the following Laplace transform identity:
\begin{align}\label{lor}
\int^{\infty}_{0}e^{-\lambda t}t^{v\rho}\phi \bigl(\rho,1+\rho v,-st^{\rho} \bigr)\, dt=\lambda^{-1-\rho v}e^{-s\lambda^{-\rho}},
\end{align}
which is valid for $s>0,$ $\Re \lambda >0$ and $1+\rho v \geq 0.$ If $0<\rho<1,$ then we know that there exist two finite real constants $c>0$ and $L>0$ such that
\begin{align}\label{asimpt}
|\phi(\rho,\nu; -r)| \leq Le^{-cr^{1/(1+\rho)}},\quad r\geq 0;
\end{align}
if $\rho=1/2,$ then the function $\phi(\rho,\nu; -r)$ can be represented in terms of the well known special functions $\mbox{erf}(r),$ $\mbox{erfc}(r)$ and $\mbox{daw}(r).$

Let $\alpha>0$ and $\beta \in {\mathbb R}.$ Then the
Mittag-Leffler function\index{Mittag-Leffler functions} $E_{\alpha,\beta}(z)$ is defined by
$$
E_{\alpha,\beta}(z):=\sum \limits_{n=0}^{\infty}\frac{z^{n}}{\Gamma(\alpha
n+\beta)},\quad z\in {\mathbb C}.
$$
Set, for short, $E_{\alpha}(z):=E_{\alpha,1}(z),$ $z\in
{\mathbb C}.$  
For more details about the Mittag-Leffler functions, the Wright functions, fractional calculus and fractional differential equations, one may refer e.g. to \cite{bajlekova}, \cite{fedorov-primonja}, \cite{gorenflo}-\cite{vkir}, \cite{knjigaho}-\cite{FKP}, \cite{luchko}-\cite{marichev} and \cite{prus}.

\subsection{Multivalued linear operators}\label{mlos}

For the introduction to the spectral theory of multivalued linear operators, we refer the reader to the monograph \cite{cross} by  R. Cross.
Applications of results from the theory of multivalued linear operators to the abstract degenerate integro-differential equations have been analyzed in the monographs \cite{faviniyagi} by A. Favini, A. Yagi, \cite{me152} by A. Filinkov, I. Melnikova and \cite{FKP} 
by M. Kosti\' c (see also the research monograph \cite{svir-fedorov}, where G. Sviridyuk and V. Fedorov have applied different methods).
We will use only the basic definitions and results from this theory.

Let $X$ and $Y$ be two SCLCSs.
A multivalued map (multimap) ${\mathcal A} : X \rightarrow P(Y)$ is said to be a multivalued
linear operator (MLO) iff the following holds:
\begin{itemize}
\item[(i)] $D({\mathcal A}) := \{x \in X : {\mathcal A}x \neq \emptyset\}$ is a linear subspace of $X$;
\item[(ii)] ${\mathcal A}x +{\mathcal A}y \subseteq {\mathcal A}(x + y),$ $x,\ y \in D({\mathcal A})$
and $\lambda {\mathcal A}x \subseteq {\mathcal A}(\lambda x),$ $\lambda \in {\mathbb C},$ $x \in D({\mathcal A}).$
\end{itemize}
If $X=Y,$ then we say that ${\mathcal A}$ is an MLO in $X.$
An almost immediate consequence of definition is that ${\mathcal A}x +{\mathcal A}y = {\mathcal A}(x+y)$ for
all $ x,\ y \in D({\mathcal A})$ and $\lambda {\mathcal A}x = {\mathcal A}(\lambda x)$ for all $x \in D({\mathcal A}),$ $ \lambda \neq 0.$ Furthermore,
for any $x,\ y\in D({\mathcal A})$ and $\lambda,\ \eta \in {\mathbb C}$ with $|\lambda| + |\eta| \neq 0,$ we
have $\lambda {\mathcal A}x + \eta {\mathcal A}y = {\mathcal A}(\lambda x + \eta y).$ If ${\mathcal A}$ is an MLO, then ${\mathcal A}0$ is a linear manifold in $Y$
and ${\mathcal A}x = f + {\mathcal A}0$ for any $x \in D({\mathcal A})$ and $f \in {\mathcal A}x.$ Set $R({\mathcal A}):=\{{\mathcal A}x :  x\in D({\mathcal A})\}.$
The set ${\mathcal A}^{-1}0 = \{x \in D({\mathcal A}) : 0 \in {\mathcal A}x\}$ is called the kernel\index{multivalued linear operator!kernel}
of ${\mathcal A}$ and it is denoted henceforth by $N({\mathcal A})$ or Kern$({\mathcal A}).$ The inverse ${\mathcal A}^{-1}$ of an MLO is defined by
$D({\mathcal A}^{-1}) := R({\mathcal A})$ and ${\mathcal A}^{-1} y := \{x \in D({\mathcal A}) : y \in {\mathcal A}x\}$.\index{multivalued linear operator!inverse}
It is checked at once that ${\mathcal A}^{-1}$ is an MLO in $X,$ as well as that $N({\mathcal A}^{-1}) = {\mathcal A}0$
and $({\mathcal A}^{-1})^{-1}={\mathcal A}.$ If $N({\mathcal A}) = \{0\},$ i.e., if ${\mathcal A}^{-1}$ is
single-valued, then ${\mathcal A}$ is said to be injective.

The sums and products of MLOs are defined in the usual way, while
the integer powers of an MLO ${\mathcal A} :  X\rightarrow P(X)$ are defined recursively, as for general binary relations.
We say that an MLO operator  ${\mathcal A} : X\rightarrow P(Y)$ is closed if for any
nets $(x_{\tau})$ in $D({\mathcal A})$ and $(y_{\tau})$ in $Y$ such that $y_{\tau}\in {\mathcal A}x_{\tau}$ for all $\tau\in I$ we have that the suppositions $\lim_{\tau \rightarrow \infty}x_{\tau}=x$ and
$\lim_{\tau \rightarrow \infty}y_{\tau}=y$ imply
$x\in D({\mathcal A})$ and $y\in {\mathcal A}x.$\index{multivalued linear operator!closed}

Now we will analyze the  $C$-resolvent sets of MLOs in locally convex spaces. Our standing assumptions will be that ${\mathcal A}$ is an MLO in $X$, $C\in L(X)$ and $C{\mathcal A}\subseteq {\mathcal A}C$  (observe that we do not require the injectiveness of $C$).
Then
the $C$-resolvent set of ${\mathcal A},$ $\rho_{C}({\mathcal A})$ for short, is defined as the union of those complex numbers
$\lambda \in {\mathbb C}$ for which
$R(C)\subseteq R(\lambda-{\mathcal A})$ and
$(\lambda - {\mathcal A})^{-1}C$ is a single-valued linear continuous operator on $X.$
The operator $\lambda \mapsto (\lambda -{\mathcal A})^{-1}C$ is called the $C$-resolvent of ${\mathcal A}$ ($\lambda \in \rho_{C}({\mathcal A})$); the resolvent set of ${\mathcal A}$ is defined by $\rho({\mathcal A}):=\rho_{I}({\mathcal A}),$ $R(\lambda : {\mathcal A})\equiv  (\lambda -{\mathcal A})^{-1}$  ($\lambda \in \rho({\mathcal A})$). 

We have
\begin{align}\label{upozori}
\bigl( \lambda-{\mathcal A} \bigr)^{-1}C{\mathcal A}\subseteq \lambda \bigl( \lambda-{\mathcal A} \bigr)^{-1}C-C\subseteq {\mathcal A}\bigl( \lambda-{\mathcal A} \bigr)^{-1}C,\quad \lambda \in \rho_{C}({\mathcal A}).
\end{align}
The operator $( \lambda-{\mathcal A})^{-1}C{\mathcal A}$ is single-valued on $D({\mathcal A})$ and the Hilbert resolvent equation
holds in our framework.

Now we will prove the following useful result:

\begin{prop}\label{lav}
Suppose that $C\in L(X),$ $\lambda \in {\mathbb C} \setminus \{0\},$ ${\mathcal A}$ is an \emph{MLO} and $\lambda^{-1} \in \rho_{C}({\mathcal A}).$ Then we have $\lambda \in \rho_{C}({\mathcal A}^{-1})$ and 
\begin{align*}
\bigl( \lambda-{\mathcal A}^{-1}\bigr)^{-1}C=\lambda^{-1}\Bigl[C-\lambda^{-1} \bigl( \lambda^{-1}- {\mathcal A}\bigr)^{-1}C \Bigr].
\end{align*} 
\end{prop}

\begin{proof}
Suppose $x\in X.$ Then a simple computation involving the definition of inverse of an MLO shows that
$$
\Bigl( Cx, \lambda^{-1}\Bigl[Cx-\lambda^{-1} \bigl( \lambda^{-1}- {\mathcal A}\bigr)^{-1}Cx \Bigr]\Bigr) \in \bigl( \lambda -{\mathcal A}^{-1} \bigr)^{-1}
$$ 
iff
$$
-Cx+\lambda^{-1}\bigl( \lambda^{-1}-{\mathcal A}\bigr)^{-1}Cx\in {\mathcal A}\bigl( \lambda^{-1}-{\mathcal A} \bigr)^{-1}Cx,
$$
which is true due to \eqref{upozori}. It suffices to prove that the operator $( \lambda-{\mathcal A}^{-1})^{-1}C$ is single-valued. If we suppose that $\{y,z\}\subseteq ( \lambda-{\mathcal A}^{-1})^{-1}Cx,$ then we have $\lambda  y-Cx\in {\mathcal A}^{-1}y$ and 
$\lambda  z-Cx\in {\mathcal A}^{-1}z.$ Hence, $y\in {\mathcal A}[\lambda  y-Cx]$ and $z\in {\mathcal A}[\lambda  z-Cx].$ This simply implies $Cx \in ( \lambda^{-1}- {\mathcal A})^{-1}C[\lambda Cx-\lambda y]$ and $Cx \in ( \lambda^{-1}- {\mathcal A})^{-1}C[\lambda Cx-\lambda z].$ Since the operator $( \lambda^{-1}- {\mathcal A})^{-1}C$ is single-valued, we simply get from the above that
$y=z=\lambda^{-1}[Cx-\lambda^{-1} ( \lambda^{-1}- {\mathcal A})^{-1}Cx].$
\end{proof}

Let us recall that for any MLO ${\mathcal A} : X \rightarrow P(Y)$ its topological closure $\overline{{\mathcal A}}$ in the space $X\times Y$ is also an MLO, with the meaning clear. Furthermore, it can be simply proved that for any two closed linear operators $A$ and $B$ in $X$ we have
\begin{align}\label{mare-care}
\overline{AB^{-1}}^{-1}=\overline{BA^{-1}}.
\end{align}

\subsection{Degenerate solution operator families}\label{solution}

In the theory of abstract ill-posed differential equations of first order, the notion of a mild $(C_{1},C_{2})$-existence
and uniqueness family plays an important role (\cite{l1}).
For the abstract Volterra integro-differential equations, we will use the following definition from \cite{FKP}. 

\begin{defn}\label{mild-akc-MLO}
Suppose $0<\tau \leq \infty,$ $k \in
C([0,\tau)),$ $k\neq 0,$ $a \in L_{loc}^{1}([0,\tau)),$ $a\neq 0,$
${\mathcal A} : X \rightarrow P(X)$ is an MLO, $C_{1}\in L(Y,X)$ and $C_{2}\in L(X).$
\begin{itemize}
\item[(i)]
Then it is said that ${\mathcal A}$ is a subgenerator\index{subgenerator} of a (local, if
$\tau<\infty$) mild $(a,k)$-regularized $(C_{1},C_{2})$-existence
and uniqueness family\index{$(a,k)$-regularized $(C_{1},C_{2})$-existence and uniqueness
family} $(R_{1}(t),R_{2}(t))_{t\in [0,\tau)}\subseteq
L(Y,X)\times L(X)$ iff the mappings $t\mapsto R_{1}(t)y,$ $t\geq 0$ and $t\mapsto R_{2}(t)x,$
$t\in [0,\tau)$ are continuous for every fixed $x\in X$ and $y\in Y,$ as well as the
following conditions hold:
\begin{equation}\label{f1}
\Biggl(\int\limits^{t}_{0}a(t-s)R_{1}(s)y\,ds, R_{1}(t)y-k(t)C_{1}y\Biggr) \in {\mathcal A},\ t\in
[0,\tau),\ y\in Y\mbox{ and}
\end{equation}
\begin{equation}\label{f2}
\int\limits^{t}_{0}a(t-s)R_{2}(s)y\,ds=R_{2}(t)x-k(t)C_{2}x ,\mbox{ whenever } t\in
[0,\tau)\mbox{ and }(x,y)\in {\mathcal A}.
\end{equation}
\item[(ii)] Let $(R_{1}(t))_{t\in [0,\tau)}\subseteq L(Y,X)$ be strongly
continuous. Then it is said that ${\mathcal A}$ is a subgenerator\index{subgenerator} of a (local,
if $\tau<\infty$) mild $(a,k)$-regularized $C_{1}$-existence family\index{$(a,k)$-regularized
$C_{1}$-existence
family}
$(R_{1}(t))_{t\in [0,\tau)}$ iff (\ref{f1})
holds.
\item[(iii)] Let $(R_{2}(t))_{t\in [0,\tau)}\subseteq L(X)$ be strongly
continuous. Then it is said that ${\mathcal A}$ is a subgenerator\index{subgenerator} of a (local,
if $\tau<\infty$) mild $(a,k)$-regularized $C_{2}$-uniqueness family
\index{$(a,k)$-regularized $C_{2}$-uniqueness family}
$(R_{2}(t))_{t\in [0,\tau)}$ iff (\ref{f2}) holds.
\end{itemize}
\end{defn}

\begin{defn}\label{coc-lock}
Suppose that $0<\tau \leq \infty,$ $k \in
C([0,\tau)),$ $k\neq 0,$ $a \in L_{loc}^{1}([0,\tau)),$ $a\neq 0,$
${\mathcal A} : X \rightarrow P(X)$ is an MLO, $C\in L(X)$ and $C{\mathcal A}\subseteq {\mathcal A}C.$ Then it is said that
a strongly continuous operator family $(R(t))_{t\in [0,\tau)}\subseteq L(X)$ is an $(a,k)$-regularized $C$-resolvent family with a subgenerator\index{subgenerator} ${\mathcal A}$ iff $(R(t))_{t\in [0,\tau)}$ is a mild $(a,k)$-regularized $C$-uniqueness family having ${\mathcal A}$ as subgenerator, $R(t)C=CR(t)$
and $R(t){\mathcal A}\subseteq {\mathcal A}R(t)$ ($t\in [0,\tau)$).
\end{defn}

Before going any further, we would like to stress that we do not require the injectiveness of regularizing operators $C,\ C_{1},\ C_{2}$ in our analysis.
Any place where the injectiveness of operators $C,\ C_{1}$ or $C_{2}$ is used will be explicitly emphasized.

An $(a,k)$-regularized $C$-resolvent family $(R(t))_{t\in [0,\tau)}$ is said to be locally equicontinuous iff, for every $t\in
(0,\tau),$ the family $\{R(s) : s\in [0,t]\}$ is equicontinuous. In
the case $\tau=\infty ,$ $(R(t))_{t\geq 0}$ is said to be
exponentially equicontinuous\index{$(a,k)$-regularized $C$-resolvent family!exponentially equicontinuous} (equicontinuous\index{$(a,k)$-regularized $C$-resolvent family!equicontinuous}) iff there exists
$\omega \in {\mathbb R}$ ($\omega =0$) such that the family $\{
e^{-\omega t} R(t) : t\geq 0\}$ is equicontinuous.
If $k(t)=g_{\alpha +1}(t),$ where $\alpha
\geq  0,$ then it is also said
that $(R(t))_{t\in [0,\tau)}$ is an $\alpha$-times integrated
$(a,C)$-resolvent family\index{$\alpha$-times integrated
$(a,C)$-resolvent family}; $0$-times integrated
$(a,C)$-resolvent family is further abbreviated to $(a,C)$-resolvent family\index{$(a,C)$-resolvent family}. 

The integral generator ${\mathcal A}_{int}$ of $(a,k)$-regularized $C$-resolvent family $(R(t))_{t\in [0,\tau)}$
is defined by
$$
{\mathcal A}_{int}:=\Biggl\{ (x,y)\in X\times X : R(t)x-k(t)Cx=\int^{t}_{0}a(t-s)R(s)y\, ds \mbox{ for all }
t\in [0,\tau)\Biggr\}.
$$
The integral generator of  mild $(a,k)$-regularized $C_{2}$-uniqueness family $(R_{2}(t))_{t\in [0,\tau)}$
is defined in the same way.
Then we have that ${\mathcal A}_{int}\subseteq C_{2}^{-1}{\mathcal A}_{int}C_{2}$ (${\mathcal A}_{int}\subseteq C^{-1}{\mathcal A}_{int}C$) is the maximal subgenerator of $(R_{2}(t))_{t\in [0,\tau)}$ ($(R(t))_{t\in [0,\tau)}$) with respect to the set inclusion and
the local equicontinuity of $(R_{2}(t))_{t\in [0,\tau)}$ ($(R(t))_{t\in [0,\tau)}$) implies that
${\mathcal A}_{int}$ is closed; moreover,
$C^{-1}{\mathcal A}_{int}C$ need not be a subgenerator of $(R(t))_{t\in [0,\tau)}$ and the inclusion $C^{-1}{\mathcal A}_{int}C\subseteq {\mathcal A}_{int}$ is not true for resolvent operator families, in general.

We need the following structural characterizations from \cite{FKP}. 

\begin{lem}\label{exp-c1c2c3-mlos}
Suppose that ${\mathcal A}$ is a closed \emph{MLO} in
$X,$ $C_{1}\in L(Y,X),$ $C_{2}\in L(X),$ $|a(t)|$ and $k(t)$ satisfy \emph{(P1)},
as well as that $(R_{1}(t),R_{2}(t))_{t\geq 0}\subseteq
L(Y,X) \times L(X)$ is strongly continuous. Let
$\omega
\geq\max(0,\emph{abs}(|a|),\emph{abs}(k))$ be such that the operator family
$\{e^{-\omega t}R_{i}(t) : t\geq 0\}$ is equicontinuous for $i=1,2.$ Then the following holds:
\begin{itemize}
\item[(i)]
$(R_{1}(t),R_{2}(t))_{t\geq 0}$ is a mild $(a,k)$-regularized
$(C_{1},C_{2})$-existence and uniqueness family with a subgenerator
${\mathcal A}$ iff
for every $\lambda \in {\mathbb C}$ with
$\Re\lambda>\omega$ and $\tilde{a}(\lambda)\tilde{k}(\lambda) \neq 0,$ we have $R(C_{1}) \subseteq
R(I-\tilde{a}(\lambda){\mathcal A}),$
\begin{equation}\label{arendt1c1c2-mlosdf}
\int
\limits_{0}^{\infty}e^{-\lambda t}R_{1}(t)y\,dt \in\tilde{k}(\lambda)\bigl(I-\tilde{a}(\lambda){\mathcal A}\bigr)^{-1}C_{1}y,\ y\in Y,
\end{equation}
and
\begin{equation}\label{arendt3-mlosdf}
\tilde{k}(\lambda)C_{2}x=\int \limits_{0}^{\infty}e^{-\lambda
t}\bigl[R_{2}(t)x-\bigl(a\ast R_{2}\bigr)(t)y\bigr]\,dt,\ \mbox{ whenever }(x,y)\in
{\mathcal A}.
\end{equation}
\item[(ii)]
$(R_{1}(t))_{t\geq 0}$ is a mild $(a,k)$-regularized
$C_{1}$-existence family with a subgenerator ${\mathcal A}$ iff for every
$\lambda \in {\mathbb C}$ with $\Re \lambda>\omega$ and
$\tilde{a}(\lambda)\tilde{k}(\lambda)\neq 0,$ one has $R(C_{1})\subseteq
R(I-\tilde{a}(\lambda){\mathcal A})$ and
$$
\tilde{k}(\lambda)C_{1}y\in \bigl(I-\tilde{a}(\lambda){\mathcal A}\bigr)\int
\limits^{\infty}_{0}e^{-\lambda t}R_{1}(t)y\,dt,\quad y\in Y.
$$
\item[(iii)] $(R_{2}(t))_{t\geq 0}$ is a mild $(a,k)$-regularized
$C_{2}$-uniqueness family with a subgenerator
${\mathcal A}$ iff \emph{(\ref{arendt3-mlosdf})} holds for $\Re \lambda>\omega.$
\end{itemize}
\end{lem}

\begin{lem}\label{exp-c1c2c3-mlos-prim}
Suppose that ${\mathcal A}$ is a closed \emph{MLO} in
$X,$ $C\in L(X),$ $C{\mathcal A}\subseteq {\mathcal A}C,$ $|a(t)|$ and $k(t)$ satisfy \emph{(P1)},
as well as that $(R(t))_{t\geq 0}\subseteq
L(X)$ is strongly continuous and commutes with $C$ on $X$. Let
$\omega
\geq\max(0,\emph{abs}(|a|),\emph{abs}(k))$ be such that the operator family
$\{e^{-\omega t}R(t) : t\geq 0\}$ is equicontinuous. Then
$(R(t))_{t\geq 0}$ is an $(a,k)$-regularized
$C$-resolvent family with a subgenerator
${\mathcal A}$ iff
for every $\lambda \in {\mathbb C}$ with
$\Re\lambda>\omega$ and $\tilde{a}(\lambda)\tilde{k}(\lambda) \neq 0,$ we have $R(C) \subseteq
R(I-\tilde{a}(\lambda){\mathcal A}),$ \emph{(\ref{arendt1c1c2-mlosdf})} holds with $R_{1}(\cdot),$ $C_{1}$ and $Y,$ $y$ replaced with $R(\cdot),$ $C$ and $X,$ $x$ therein, as well as
\emph{(\ref{arendt3-mlosdf})} holds with $R_{2}(\cdot)$ and $C_{2}$ replaced with $R(\cdot)$ and $C$ therein.
\end{lem}

\begin{defn}\label{3.1} (\cite{FKP})
\begin{itemize}
\item[(i)] Suppose that ${\mathcal A}$ is an MLO in $X.$
Let $\alpha \in (0,\pi],$ and
let $(R(t))_{t\geq 0}$ be an $(a,k)$-regularized $C$-resolvent
family which do have ${\mathcal A}$ as a subgenerator. Then it is said that $(R(t))_{t\geq 0}$ is an analytic
$(a,k)$-regularized $C$-resolvent family\index{$(a,k)$-regularized $C$-resolvent family!analytic} of angle $\alpha,$ if
there exists a function ${\bf R} : \Sigma_{\alpha} \rightarrow L(X)$
which satisfies that, for every $x\in X,$ the mapping $z\mapsto {\bf
R}(z)x,$ $z\in \Sigma_{\alpha}$ is analytic as well as that:
\begin{itemize}
\item[(a)] ${\bf R}(t)=R(t),\ t>0$ and
\item[(b)] $\lim_{z\rightarrow
0,z\in \Sigma _{\gamma }}{\bf R}(z)x=R(0)x$ for
all $\gamma \in (0,\alpha)$ and $x\in X.$
\end{itemize}
\item[(ii)] Let $(R(t))_{t\geq 0}$ be an analytic
$(a,k)$-regularized $C$-resolvent family\index{$(a,k)$-regularized $C$-resolvent family!analytic} of angle $\alpha \in (0,\pi].$
Then it is said that $(R(t))_{t\geq 0}$ is an exponentially
equicontinuous, analytic\index{$(a,k)$-regularized $C$-resolvent family!exponentially equicontinuous, analytic}
$(a,k)$-regularized $C$-resolvent family of angle $\alpha ,$
resp. equicontinuous analytic $(a,k)$-regularized
$C$-resolvent family of angle $\alpha ,$\index{$(a,k)$-regularized $C$-resolvent family!equicontinuous, analytic} if for every
$\gamma\in(0,\alpha),$ there exists $ \omega_{\gamma} \geq 0,$ resp.
$\omega_{\gamma}=0,$ such that the family $\{e^{-\omega_{\gamma}\Re
z}{\bf R}(z) : z\in \Sigma_{\gamma}\}\subseteq L(X)$ is equicontinuous. Since
there is no risk for confusion, we will identify in the sequel
$R(\cdot)$ and ${\bf R}(\cdot).$
\end{itemize}
\end{defn}

The following structrural result for exponentially
equicontinuous, analytic $(a,k)$-regularized $C$-resolvent families plays an important role in our study (\cite{FKP}):

\begin{lem}\label{3.7-MLOpope}
Assume that ${\mathcal A}$ is a
closed \emph{MLO} in $X,$ $C{\mathcal A}\subseteq {\mathcal A}C,$
$\alpha \in (0,
\pi /2],$ $\emph{abs}(k)<\infty,$ $\emph{abs}(|a|)<\infty$ and $\omega \geq \max(0, \emph{abs}(k),\emph{abs}(|a|)).$
Assume, further, that for every $\lambda \in {\mathbb C}$
with $\Re \lambda>\omega$ and $\tilde{a}(\lambda)\tilde{k}(\lambda)\neq 0,$ we have
$R(C)\subseteq R(I-\tilde{a}(\lambda){\mathcal A})$ as well as that there
exist a function $q :
\omega+\Sigma_{\frac{\pi}{2}+\alpha}\rightarrow L(X)$ and an operator $D\in L(X)$ such that, for
every $x\in X,$ the mapping $\lambda \mapsto q(\lambda)x,$ $\lambda
\in \omega+\Sigma_{\frac{\pi}{2}+\alpha}$ is analytic as well as
that:
\begin{itemize}
\item[(i)] $q(\lambda)x\in \tilde{k}(\lambda)(I-\tilde{a}(\lambda){\mathcal A})^{-1}Cx$ for
$\Re\lambda>\omega_{0},$ $\tilde{a}(\lambda)\tilde{k}(\lambda)\neq 0,$ $x\in X;$
\item[(ii)]
$q(\lambda)x-\tilde{a}(\lambda)q(\lambda)y=\tilde{k}(\lambda)Cx$ for $\Re\lambda>\omega_{0}$ and
$(x,y)\in {\mathcal A};$ 
\item[(iii)] $q(\lambda)Cx=Cq(\lambda)x$ for $\lambda \in \omega+\Sigma_{\frac{\pi}{2}+\alpha}$ and
$x\in X;$
\item[(iv)] the family
$\{(\lambda-\omega)q(\lambda) : \lambda \in
\omega+\Sigma_{\frac{\pi}{2}+\gamma} \}\subseteq L(X)$ is equicontinuous for
all $\gamma \in (0, \alpha);$
\item[(v)]
$
\lim_{\lambda \rightarrow +\infty}\lambda q(\lambda)x=Dx,\ x\in
X, \mbox{ if }\overline{D({\mathcal A})}\neq X.
$
\end{itemize}
Then ${\mathcal A}$ is a subgenerator of an exponentially equicontinuous,\index{$(a,k)$-regularized $C$-resolvent family!analytic}
analytic\\ $(a,k)$-regularized $C$-resolvent family $(R(t))_{t\geq 0}$
of angle $\alpha $ satisfying that $R(z){\mathcal A} \subseteq {\mathcal A}R(z),$
$z\in \Sigma_{\alpha},$ the family
$\{e^{-\omega z}R(z) : z\in \Sigma_{\gamma}\}\subseteq L(X)$ is
equicontinuous for all angles $\gamma \in (0,\alpha),$ as well as that the equation \emph{(\ref{f1})}
holds for each $y=x\in X,$ with $R_{1}(\cdot)$ and $C_{1}$ replaced therein by $R(\cdot)$ and $C,$ respectively.
\end{lem}

\section{Inverse generator problem for degenerate integro-differential equations}\label{mario}

We start by stating a general result which gives the necessary and sufficient conditions for a multivalued linear operator ${\mathcal A}^{-1}$ to be a subgenerator of an exponentially equicontinuous mild $(a,k)$-regularized
$C_{1}$-existence family or an exponentially equicontinuous mild $(a,k)$-regularized
$C_{2}$-uniqueness family.

\begin{prop}\label{loza}
Suppose that ${\mathcal A}$ is a closed \emph{MLO} in
$X,$ $C_{1}\in L(Y,X),$ $C_{2}\in L(X),$ $|a(t)|$ and $k(t)$ satisfy \emph{(P1)},
as well as that $(R_{1}(t),R_{2}(t))_{t\geq 0}\subseteq
L(Y,X) \times L(X)$ is strongly continuous. Let
$\omega
\geq\max(0,\emph{abs}(|a|),\emph{abs}(k))$ be such that the operator family
$\{e^{-\omega t}R_{i}(t) : t\geq 0\}$ is equicontinuous for $i=1,2.$ Then the following holds:
\begin{itemize}
\item[(i)]
$(R_{1}(t),R_{2}(t))_{t\geq 0}$ is a mild $(a,k)$-regularized
$(C_{1},C_{2})$-existence and uniqueness family with a subgenerator
${\mathcal A}^{-1}$ iff
for every $\lambda \in {\mathbb C}$ with
$\Re\lambda>\omega$ and $\tilde{a}(\lambda)\tilde{k}(\lambda) \neq 0,$ we have $R(C_{1}) \subseteq
R(\tilde{a}(\lambda)-{\mathcal A}),$
\begin{equation}\label{arenq}
\tilde{a}(\lambda)\tilde{k}(\lambda) C_{1}y\in \bigl(\tilde{a}(\lambda)-{\mathcal A}\bigr)\int
\limits^{\infty}_{0}e^{-\lambda t}\bigl[k(t)C_{1}y-R_{1}(t)y\bigr]\,dt,\quad y\in Y
\end{equation}
and
\begin{equation}\label{aren}
\tilde{k}(\lambda)C_{2}x=\int \limits_{0}^{\infty}e^{-\lambda
t}\bigl[R_{2}(t)x-\bigl(a\ast R_{2}\bigr)(t)y\bigr]\,dt,\ \mbox{ whenever }(y,x)\in
{\mathcal A}.
\end{equation}
\item[(ii)]
$(R_{1}(t))_{t\geq 0}$ is a mild $(a,k)$-regularized
$C_{1}$-existence family with a subgenerator ${\mathcal A}^{-1}$ iff for every
$\lambda \in {\mathbb C}$ with $\Re \lambda>\omega$ and
$\tilde{a}(\lambda)\tilde{k}(\lambda)\neq 0,$ we have $R(C_{1})\subseteq
R(\tilde{a}(\lambda)-{\mathcal A})$ and \eqref{arenq}.
\item[(iii)] $(R_{2}(t))_{t\geq 0}$ is a mild $(a,k)$-regularized
$C_{2}$-uniqueness family with a subgenerator
${\mathcal A}^{-1}$ iff \emph{(\ref{aren})} holds for $\Re \lambda>\omega.$
\end{itemize}
\end{prop}

\begin{proof}
It is clear that ${\mathcal A}^{-1}$ is a closed MLO. The part (iii) follows immediately from Lemma \ref{exp-c1c2c3-mlos} and definition of ${\mathcal A}^{-1}.$ For the rest, it suffices to prove (ii). Suppose first that, for every
$\lambda \in {\mathbb C}$ with $\Re \lambda>\omega$ and
$\tilde{a}(\lambda)\tilde{k}(\lambda)\neq 0,$ we have $R(C_{1})\subseteq
R(\tilde{a}(\lambda)-{\mathcal A})$ and \eqref{arenq}. Then a simple computation gives that for such values of parameter $\lambda$ we have
$$
-\int^{\infty}_{0}e^{-\lambda t}\bigl(a\ast R_{1}\bigr)(t)y\, dt \in {\mathcal A}\int^{\infty}_{0}e^{-\lambda t}\bigl[ k(t)C_{1}y-R_{1}(t)y \bigr]\, dt,\quad y\in Y,
$$
$R(C_{1})\subseteq R(I-\tilde{a}(\lambda){\mathcal A}^{-1})$ and
\begin{align}\label{ers}
\tilde{k}(\lambda)C_{1}y\in \Bigl( I-\tilde{a}(\lambda){\mathcal A}^{-1}\Bigr)\int^{\infty}_{0}e^{-\lambda t}R_{1}(t)y\, dt,\quad y\in Y.
\end{align}
By Lemma \ref{exp-c1c2c3-mlos}(ii), we get that $(R_{1}(t))_{t\geq 0}$ is a mild $(a,k)$-regularized
$C_{1}$-existence family with a subgenerator ${\mathcal A}^{-1}.$ For the converse, we can apply Lemma \ref{exp-c1c2c3-mlos}(ii) again so as to conclude that, for every $y\in Y$ and for every
$\lambda \in {\mathbb C}$ with $\Re \lambda>\omega$ and
$\tilde{a}(\lambda)\tilde{k}(\lambda)\neq 0,$ we have $R(C_{1})\subseteq R(I-\tilde{a}(\lambda){\mathcal A}^{-1})$ and \eqref{ers}. As above, this simply implies $R(C_{1})\subseteq
R(\tilde{a}(\lambda)-{\mathcal A})$ and \eqref{arenq}.
\end{proof}

Using Lemma \ref{exp-c1c2c3-mlos-prim}
and a similar argumentation, we can prove the following:

\begin{prop}\label{expaq}
Suppose that ${\mathcal A}$ is a closed \emph{MLO} in
$X,$ $C\in L(X),$ $C{\mathcal A}\subseteq {\mathcal A}C,$ $|a(t)|$ and $k(t)$ satisfy \emph{(P1)},
as well as that $(R(t))_{t\geq 0}\subseteq
L(X)$ is strongly continuous and commutes with $C$ on $X$. Let
$\omega
\geq\max(0,\emph{abs}(|a|),\emph{abs}(k))$ be such that the operator family
$\{e^{-\omega t}R(t) : t\geq 0\}$ is equicontinuous. Then
$(R(t))_{t\geq 0}$ is an $(a,k)$-regularized
$C$-resolvent family with a subgenerator
${\mathcal A}^{-1}$ iff
for every $\lambda \in {\mathbb C}$ with
$\Re\lambda>\omega$ and $\tilde{a}(\lambda)\tilde{k}(\lambda) \neq 0,$ we have $R(C) \subseteq
R(\tilde{a}(\lambda)-{\mathcal A}),$ \emph{(\ref{arenq})} holds with $R_{1}(\cdot),$ $C_{1}$ and $Y,$ $y$ replaced with $R(\cdot),$ $C$ and $X,$ $x$ therein, as well as
\emph{(\ref{aren})} holds with $R_{2}(\cdot)$ and $C_{2}$ replaced with $R(\cdot)$ and $C$ therein.
\end{prop}

The complex characterization theorem ensuring the existence of an exponentially equicontinuous $(a,k)$-regularized $C$-resolvent family subgenerated by the multivalued linear operator ${\mathcal A}$ has been recently clarified in \cite{FKP}. Due to Proposition \ref{expaq}, we can simply formulate a corresponding statement for the inverse operator ${\mathcal A}^{-1}.$
We will explain this in more detail 
in the proof of subsequent result, which provides sufficient conditions for the operator ${\mathcal A}^{-1}$ to be a subgenerator of an exponentially equicontinuous, analytic $(a,k)$-regularized $C$-resolvent family (see also Lemma \ref{3.7-MLOpope}).

\begin{prop}\label{pope}
Assume that ${\mathcal A}$ is a
closed \emph{MLO} in $X,$ $C{\mathcal A}\subseteq {\mathcal A}C,$
$\alpha \in (0,
\pi /2],$ $\emph{abs}(k)<\infty,$ $\emph{abs}(|a|)<\infty$ and $\omega \geq \max(0, \emph{abs}(k),\emph{abs}(|a|)).$
Assume, further, that for every $\lambda \in {\mathbb C}$
with $\Re \lambda>\omega$ and $\tilde{a}(\lambda)\neq 0,$ we have
$R(C)\subseteq R(\tilde{a}(\lambda) -{\mathcal A})$ as well as that there
exist a function $\Upsilon :
\omega+\Sigma_{\frac{\pi}{2}+\alpha}\rightarrow L(X)$ and an operator $D'\in L(X)$ such that, for
every $x\in X,$ the mapping $\lambda \mapsto \Upsilon(\lambda)x,$ $\lambda
\in \omega+\Sigma_{\frac{\pi}{2}+\alpha}$ is analytic as well as
that:
\begin{itemize}
\item[(i)] There exists a function ${\bf k} :\Sigma_{\alpha} \cup \{0\} \rightarrow {\mathbb C}$ which is analytic on $\Sigma_{\alpha},$
continuous on any closed subsector $\overline{\Sigma_{\gamma}}$ ($0<\gamma<\alpha$) and which additionally satisfies that $\sup_{z\in \Sigma_{\gamma}}|e^{-\omega z}{\bf k}(z)|<\infty$ ($0<\gamma<\alpha$) and
${\bf k}(t)=k(t)$ for all $t\geq 0;$
\item[(ii)] $\Upsilon(\lambda)x \in \tilde{a}(\lambda)(\tilde{a}(\lambda)-A)^{-1}Cx$ for every $x\in X$ and $\lambda \in {\mathbb C}$ with $\Re \lambda >\omega ,$ $\tilde{a}(\lambda)\neq 0;$  
\item[(iii)] $\Upsilon(\lambda)Cx=C\Upsilon(\lambda)x$ for $\Re \lambda >\omega ,$ $x\in X;$
\item[(iv)] $\tilde{a}(\lambda)\Upsilon(\lambda)x -\Upsilon(\lambda)y=\tilde{a}(\lambda)Cx ,$ provided $\Re \lambda >\omega $ and $(x,y)\in {\mathcal A};$
\item[(v)] $
\lim_{\lambda \rightarrow +\infty}\Upsilon (\lambda)x=D'x,\ x\in
X, \mbox{ if }\overline{R({\mathcal A})}\neq X.
$
\end{itemize}
Then the function $\lambda \mapsto \tilde{k}(\lambda),$ $\Re \lambda>\omega$ has an analytic extension $\lambda \mapsto \hat{k}(\lambda),$ $\lambda \in \omega +\Sigma_{(\pi/2)+\alpha}$ satisfying that $\sup_{\omega +\Sigma_{(\pi/2)+\gamma}}|(\lambda -\omega )\hat{k}(\lambda)|<\infty$ for $0<\gamma<\alpha$
and $\lim_{\lambda \rightarrow +\infty}\lambda\hat{k}(\lambda)=k(0).$ If, additionally,
\begin{itemize}
\item[(vi)]the family
$\{(\lambda-\omega)\hat{k}(\lambda)\Upsilon(\lambda) : \lambda \in
\omega+\Sigma_{\frac{\pi}{2}+\gamma} \}\subseteq L(X)$ is equicontinuous for
all $\gamma \in (0, \alpha),$
\end{itemize}
then
${\mathcal A}^{-1}$ is a subgenerator of an exponentially equicontinuous,\index{$(a,k)$-regularized $C$-resolvent family!analytic}
analytic\\ $(a,k)$-regularized $C$-resolvent family $(R(t))_{t\geq 0}$
of angle $\alpha $ satisfying that $R(z){\mathcal A}^{-1} \subseteq {\mathcal A}^{-1}R(z),$
$z\in \Sigma_{\alpha},$ the family
$\{e^{-\omega z}R(z) : z\in \Sigma_{\gamma}\}\subseteq L(X)$ is
equicontinuous for all angles $\gamma \in (0,\alpha),$ as well as that the equation \emph{(\ref{f1})}
holds for each $y=x\in X,$ with ${\mathcal A},$ $R_{1}(\cdot)$ and $C_{1}$ replaced therein by ${\mathcal A}^{-1},$ $R(\cdot)$ and $C,$ respectively.
\end{prop}

\begin{proof}
It is clear that ${\mathcal A}^{-1}$ is a
closed MLO in $X$ and $C{\mathcal A}^{-1}\subseteq {\mathcal A}^{-1}C.$
By condition (i) and \cite[Theorem 2.6.1, Theorem 2.6.4 a)]{a43}, we get that the function $\lambda \mapsto \tilde{k}(\lambda),$ $\Re \lambda>\omega$ has an analytic extension $\lambda \mapsto \hat{k}(\lambda),$ $\lambda \in \omega +\Sigma_{(\pi/2)+\alpha}$ satisfying that $\sup_{\omega +\Sigma_{(\pi/2)+\gamma}}|(\lambda -\omega )\hat{k}(\lambda)|<\infty$ for $0<\gamma<\alpha$
and $\lim_{\lambda \rightarrow +\infty}\lambda\hat{k}(\lambda)=k(0).$
Define $q(\lambda):=\hat{k}(\lambda)C-\hat{k}(\lambda)\Upsilon(\lambda),$ $\lambda \in \omega +\Sigma_{(\pi/2)+\alpha}$ and $D:=k(0)C-k(0)D'.$ Then $q(\cdot)$ is analytic
and a simple computation involving condition (ii) shows that for every $\lambda \in {\mathbb C}$
with $\Re \lambda>\omega$ and $\tilde{a}(\lambda)\neq 0,$ we have
$R(C)\subseteq R(I-\tilde{a}(\lambda){\mathcal A}^{-1})$ with $Cx\in (I-\tilde{a}(\lambda){\mathcal A}^{-1})[Cx-\Upsilon(\lambda)x],$ $x\in X.$
Therefore, it suffices to show that conditions (i)-(v) of Lemma \ref{3.7-MLOpope} hold true with the operator ${\mathcal A}$ replaced with the operator ${\mathcal A}^{-1}$ therein. It is clear that (iii) holds since $q(\cdot)$ commutes with $C,$ due to condition (iii) made in the formulation of this proposition. Conditions (iv) and (v) can be simply verified. Condition (i) of Lemma \ref{3.7-MLOpope} simply follows from condition (ii) of this proposition and simple computation, while condition (ii) simply follows from condition (iv) of this proposition. The proof of the proposition is thereby complete.
\end{proof}

The most intriguing case in which the assumptions of Proposition \ref{pope} hold is: $\omega=0,$ $a(t)=g_{\alpha}(t)$ for some number $\alpha \in (0,2),$ $k(t)=1,$ $C\in L(X)$ is injective and satisfies that $C^{-1}{\mathcal A}C={\mathcal A}$ is the integral generator of an equicontinuous analytic $(g_{\alpha},g_{1})$-regularized $C$-resolvent family $(S(t))_{t\geq 0}$ of angle $\gamma \in (0,\min(\pi/2,(\pi/\alpha)-(\pi/2))].$ Then, due to \cite[Theorem 3.2.18]{FKP}, we can make a choice in which $\Upsilon(\lambda)=\lambda^{\alpha}(\lambda^{\alpha}-{\mathcal A})^{-1}C$ and $D'=S(0),$ providing thus a proper extension of \cite[Theorem 4.1(i)]{miao-li} and \cite[Proposition 1]{l1111}. Since subordination principles can be formulated for $(a,k)$-regularized $C$-resolvent families subgenerated by MLOs ($C$ need not be injective, in general),
we can simply prove an extension of \cite[Theorem 4.1(ii)]{miao-li} for $(g_{\alpha},g_{\beta})$-regularized $C$-resolvent families.

In connection with Proposition \ref{pope} and \cite[Theorem 4.1]{miao-li}, we would like to propose the following:

\begin{prop}\label{popep}
Suppose that ${\mathcal A}$ is a closed \emph{MLO,} $C\in L(X),$ $C{\mathcal A}\subseteq {\mathcal A}C,$ $\beta \geq 0,$ $\alpha \in (0,2),$ $a(t)=g_{\alpha}(t),$ $k(t)=g_{\beta +1}(t)$ and $\gamma \in (0,\min(\pi/2,(\pi/\alpha)-(\pi/2))].$ Suppose, further, that $\Sigma_{((\pi/2)+\gamma)\alpha}\subseteq \rho_{C}({\mathcal A}),$ the mapping $\lambda \mapsto (\lambda-{\mathcal A})^{-1}Cx,$ $\lambda \in \Sigma_{((\pi/2)+\gamma)\alpha}$ is analytic ($x\in X$) and the following two conditions are satisfied:
\begin{itemize}
\item[(i)] For every $\gamma' \in (0,\gamma),$ there exists a
finite constant $M_{\gamma'}>0$ such that the operator family
$
\{ \lambda^{\alpha+\beta}(\lambda^{\alpha}-{\mathcal A})^{-1}C : \lambda \in \Sigma_{(\pi/2)+\gamma'},\ |\lambda| \leq 1\}\subseteq L(X)
$ is equicontinuous.
\item[(ii)] If $\overline{R({\mathcal A})}\neq X,$ then there exists $D'\in L(X)$ such that
$
\lim_{\lambda \rightarrow 0+}\lambda (\lambda-{\mathcal A})^{-1}Cx=D'x,\ x\in
X.
$
\end{itemize}
Then the operator ${\mathcal A}^{-1}$ is a subgenerator of an exponentially equicontinuous,\index{$(a,k)$-regularized $C$-resolvent family!analytic}
analytic $(a,k)$-regularized $C$-resolvent family $(R(t))_{t\geq 0}$
of angle $\gamma $ satisfying that\\ $R(z){\mathcal A}^{-1} \subseteq {\mathcal A}^{-1}R(z),$
$z\in \Sigma_{\gamma}$ and the family
$\{e^{-\omega z}R(z) : z\in \Sigma_{\gamma}\}\subseteq L(X)$ is
equicontinuous for every real numbers $\omega >0$ and $\gamma \in (0,\gamma').$ Moreover, the equation \emph{(\ref{f1})}
holds for each $y=x\in X,$ with ${\mathcal A},$ $R_{1}(\cdot)$ and $C_{1}$ replaced therein by ${\mathcal A}^{-1},$ $R(\cdot)$ and $C,$ respectively.
\end{prop}

\begin{proof}
As above, we have that  ${\mathcal A}^{-1}$ is a
closed MLO in $X$ and $C{\mathcal A}^{-1}\subseteq {\mathcal A}^{-1}C.$ It is clear that the function $k(t)$ satisfies condition (i) from Proposition \ref{pope}. 
If $\omega >0,$ then we can apply Proposition \ref{pope} with the function $\Upsilon :
\omega+\Sigma_{\frac{\pi}{2}+\gamma}\rightarrow L(X)$ defined by  $\Upsilon (\lambda):=\lambda^{-\alpha}(\lambda^{-\alpha}-{\mathcal A})^{-1}C,$ $\lambda \in \omega+\Sigma_{\frac{\pi}{2}+\gamma}.$ In actual fact, conditions (ii)-(iv) of Proposition \ref{pope} clearly hold; condition (v) of Proposition \ref{pope} holds
because of assumption (ii) of this proposition, while condition (vi) of Proposition \ref{pope} follows from condition (i) of this proposition and a simple computation with a new variable $z=1/\lambda.$
\end{proof}

\begin{rem}\label{cjeobroj}
Note that the behaviour of function $\lambda \mapsto (\lambda-{\mathcal A})^{-1}C,$ $\lambda \in \Sigma_{((\pi/2)+\gamma)\alpha}$ at the point $\lambda=\infty$ does not play any role for applying Proposition \ref{popep}.
\end{rem}

It seems that Proposition \ref{popep} is not considered elsewhere, even for the abstract non-degenerate differential equations of first order. We will only present an illustrative application of Proposition \ref{popep} to the abstract degenerate differential equations of fractional order:

\begin{example}\label{jove}
It is clear that, for every two linear single-valued operators $A$ and $B,$ we have $(AB^{-1})^{-1}=BA^{-1}$
and $(B^{-1}A)^{-1}=A^{-1}B$ in the MLO sense. This is very important to be noted since
many authors have investigated infinitely differentiable semigroups\index{semigroup!infinitely differentiable} generated by multivalued linear operators of form $AB^{-1}$ or $B^{-1}A,$ where the operators $A$ and $B$ satisfy the condition \cite[(3.7)]{faviniyagi}, or its slight modification. Consider, for the sake of illustration, the following fractional Poisson heat equation in the space $L^{p}(\Omega):$
\index{equation!backward Poisson heat}
\[(P)_{b}:\left\{
\begin{array}{l}
{\mathbf D}_{t}^{\alpha} [m(x)v(t,x)]=(\Delta  -b)v(t,x),\quad t\geq 0,\ x\in {\Omega};\\
v(t,x)=0,\quad (t,x)\in [0,\infty) \times \partial \Omega ,\\
 m(x)v(0,x)=u_{0}(x),\ \Bigl(\frac{d}{dt}[m(x)v(t,x)]\Bigr)_{t=0}=u_{1}(x),\quad x\in {\Omega},
\end{array}
\right.
\]
where $\Omega$ is a bounded domain in ${\mathbb R}^{n},$ $1<p<\infty,$ $b>0,$ $m(x)\geq 0$ a.e. $x\in \Omega$, $m\in L^{\infty}(\Omega)$ and $1<p<\infty;$ here, $B$ is the multiplication in $L^{p}(\Omega)$ with $m(x),$ and $A=\Delta -b$ acts with the Dirichlet boundary conditions (see \cite[Example 3.6]{faviniyagi}). Let ${\mathcal A}:=AB^{-1};$ then for a suitable chosen number $b>0,$ we have the existence of an angle $\theta \in (\pi /2,\pi)$ and a finite number $M>0$ such that 
\begin{align}\label{ras}
\bigl\| ( \lambda -{\mathcal A})^{-1} \bigr\|\leq M|\lambda|^{(-1)/p},
\quad  \lambda \in \Sigma_{\theta}.
\end{align}
Let $\alpha \in [1,2\theta/\pi).$
By Proposition \ref{popep}, with $D'= 0$ and $\beta = 0,$ we get that ${\mathcal A}^{-1}$ generates an 
analytic $(g_{\alpha},g_{1})$-regularized resolvent family $(R(t))_{t\geq 0}$
of angle $\gamma \in (0,\min((\theta/\alpha)-(\pi/2),\pi/2)],$ satisfying that the operator family
$\{e^{-\omega z}R(z) : z\in \Sigma_{\gamma'}\}\subseteq L(X)$ is
bounded for every $\omega>0$ and $\gamma' \in (0,\gamma).$ Moreover,
let $0<\epsilon<\gamma'<\gamma.$ Then we have the following integral representation
$$
R(z)x=x-\frac{1}{2\pi i}\int_{\Gamma_{\omega}}e^{\lambda z}\frac{\bigl( \lambda^{-\alpha}-{\mathcal A}\bigr)^{-1}x}{\lambda^{\alpha+1}}\, d\lambda,\quad x\in X,\ z\in \Sigma_{\gamma'-\epsilon},
$$
where the contour $\Gamma$ is defined in the proof of \cite[Theorem 2.6.1]{a43} (we only need to replace the number $\gamma$ with the number $\gamma'$ therein).
Using the estimate \eqref{ras} and the integral computation contained in the proof of afore-mentioned theorem, after letting $\omega \rightarrow 0+$ we get that $(R(t))_{t\geq 0}$ is an equicontinuous analytic $(g_{\alpha},g_{1})$-regularized resolvent family of angle $\gamma.$
Hence,
we can analyze the well-posedness of the reversed fractional Poisson heat equation in the space $L^{p}(\Omega)$:\index{equation!backward Poisson heat}
\[(P)_{b;r}:\left\{
\begin{array}{l}
{\mathbf D}_{t}^{\alpha} [(\Delta  -b)v(t,x)]=m(x)v(t,x),\quad t\geq 0,\ x\in {\Omega};\\
v(t,x)=0,\quad (t,x)\in [0,\infty) \times \partial \Omega ,\\
 (\Delta -b)v(0,x)=v_{0}(x),\ \Bigl(\frac{d}{dt}[(\Delta  -b)v(t,x)]\Bigr)_{t=0}=v_{1}(x), \quad x\in {\Omega}.
\end{array}
\right.
\]
For possible applications to abstract degenerate second-order differential equations, we refer the reader to \cite[Example 6.1]{faviniyagi} and \cite[Example 3.10.10]{nova-mono}. 
\end{example}

It is worth noticing that the existence and behaviour of $C$-resolvent of a multivalued linear operator ${\mathcal A}$ around zero is most important for the generation of certain classes of 
$(a,k)$-regularized $C$-resolvent families, $C$-(ultra)distribution semigroups and $C$-(ultra)distribution cosine functions by the inverse operator ${\mathcal A}^{-1}$ (see \cite[Section 3.3, Section 3.4]{FKP} for the notion and further information on the subject). More to the point, the existence of $C$-resolvent of ${\mathcal A}$ at the point $\lambda=+\infty$ does not play any role for the generation of $C$-(ultra)distribution semigroups and $C$-(ultra)distribution cosine functions by the inverse operator ${\mathcal A}^{-1};$ in the following example, we will explain this fact only for $C$-distribution semigroups (a similar statement holds for the generation of locally defined fractional $C$-resolvent families):

\begin{example}\label{distsemig}
Suppose that $a>0$ and $ b>0.$
The exponential region $E(a,b)$ was defined by W. Arendt, O. El--Mennaoui and V. Keyantuo in 1994 (\cite{a22}): 
$$
E(a,b):=\Bigl\{\lambda\in\mathbb{C} \ | \ \Re\lambda\geq b,\:|\Im\lambda|\leq e^{a\Re\lambda}\Bigr\}.
$$
It can be easily seen that the set $1/E(a,b):=\{1/\lambda : \lambda \in E(a,b)\}$ is a relatively compact subset of ${\mathbb C},$ as well as that $1/E(a,b)$ is contained in the strip $\{\lambda \in {\mathbb C} : 0<\Re \lambda <1/b\}.$ Let ${\mathcal A}$ be a closed MLO commuting with the operator $C\in L(X)$, and let there exist $n\in {\mathbb N}$ such that the operator family $\{\lambda^{n}(\lambda-{\mathcal A})^{-1}C : \lambda \in 1/E(a,b)\}\subseteq L(X)$ is equicontinuous. If we suppose additionally that the mapping $\lambda \mapsto (\lambda -{\mathcal A})^{-1}C x$ is analytic on $1/\Omega_{a,b}$ and continuous on $1/\Gamma_{a,b},$ with the meaning clear,
where $\Gamma_{a,b}$ denotes the upwards oriented boundary of $E(a,b)$
and $\Omega_{a,b}$ the open region which lies to the right of $\Gamma_{a,b},$ 
then we can apply Proposition \ref{lav} and \cite[Theorem 3.3.15]{FKP} to conclude that an extension of the operator ${\mathcal A}^{-1}$ generates a $C$-distribution semigroup ${\mathcal G}.$ Moreover, if $C$ is injective and ${\mathcal A}^{-1}$ is single valued, then the operator $C^{-1}{\mathcal A}^{-1}C$ is the integral generator of ${\mathcal G}$ (see \cite[Remark 3.3.16]{FKP}). On the other hand, for the generation of exponential $C$-distribution semigroups by an extension of the operator ${\mathcal A}^{-1}$ one has to assume that the operator family $\{\lambda^{n}(\lambda-{\mathcal A})^{-1}C : 0<\Re \lambda <c \}\subseteq L(X)$ is equicontinuous for some real number $c>0$ and integer $n\in {\mathbb N}.$ Finally, let $\alpha \in (0,2)$ and $\omega>0.$ Denote by $\Omega$ the unbounded region lying between the boundary of sector $\Sigma_{\alpha \pi/2}$ and the curve $\{\lambda^{-\alpha}: \Re \lambda =\omega \}.$  If $\Omega \subseteq \rho_{C}({\mathcal A})$ and the family  $\{\lambda^{n}(\lambda-{\mathcal A})^{-1}C :  \lambda \in \Omega \}\subseteq L(X)$ is equicontinuous for some integer $n\in {\mathbb N},$ then there exists a positive real number $\beta > 0$ such that the operator ${\mathcal A}$ is a subgenerator of a global $(g_{\alpha},g_{\beta+1})$-regularized $C$-resolvent operator family $(R(t))_{t\geq 0}$ satisfying that the operator family $\{e^{-\omega t}R(t) : t\geq 0\}\subseteq L(X)$ is equicontinuous.
\end{example}

In the following theorem, we reconsider the statement of \cite[Theorem 4.1(i)]{miao-li} for subgenerators of degenerate 
$(g_{\alpha},g_{\beta+1})$-regularized $C$-resolvent families, where $\alpha \in (0,2],$ $\beta \geq 0$ and the operator $C\in L(X)$ is possibly non-injective. More to the point, we consider the situation in which the subgenerator ${\mathcal A}$ is not necessarily injective or single-valued, in the setting of general SCLCSs (the interested reader may try to extend the statements of \cite[Theorem 4.1(ii), Corollary 4.1]{miao-li} in this framework), which will be crucial for applications carried out in Example \ref{zvone} below:

\begin{thm}\label{prostaku}
Suppose that $\alpha \in (0,2],$ $\beta \geq 0$ and a closed \emph{MLO} ${\mathcal A}$ is a subgenerator of an exponentially equicontinuous $(g_{\alpha},g_{\beta+1})$-regularized $C$-resolvent family $(S(t))_{t\geq 0}$ such that the operator family $\{t^{-\beta}S(t) : t>0\}\subseteq L(X)$ is equicontinuous.  Then, for every number $\gamma>\beta +(1/2),$ the operator ${\mathcal A}^{-1}$ is a subgenerator of an  $(g_{\alpha},g_{\gamma+1})$-regularized $C$-resolvent family $(R(t))_{t\geq 0}$ satisfying that the operator family $\{t^{-\gamma}R(t) : t>0\}\subseteq L(X)$ is equicontinuous.
\end{thm}

\begin{proof}
Define 
$$
R(t)x:=g_{\gamma +1}(t)Cx-t^{1+\beta +\gamma}\int^{\infty}_{0}J_{1+\beta+\gamma}\bigl( 2\sqrt{st} \bigr) s^{-\frac{1+\beta +\gamma}{2}}S(s)x\, ds,\quad t>0,\ x\in X.
$$
Arguing as in the proof of \cite[Theorem 4.1(i)]{miao-li}, we get that 
$(R(t))_{t\geq 0}\subseteq L(X)$ is strongly continuous as well as that 
the operator family $\{t^{-\gamma}R(t) : t>0\}\subseteq L(X)$ is equicontinuous and
\begin{align}\label{lipo-preb}
\int^{\infty}_{0}e^{-\lambda t}R(t)x\, dt=\lambda^{-(1+\gamma)}Cx-\lambda^{-(2+\beta +\gamma)}\int^{\infty}_{0}e^{-s/\lambda}S(s)x\, ds,\quad \Re \lambda>0,\ x\in X.
\end{align}
Further on, ${\mathcal A}^{-1}$ is a closed MLO and, by Lemma \ref{exp-c1c2c3-mlos-prim}, we have
\begin{align}\label{ola}
\frac{Cx}{\lambda^{\beta+1}}\in \Biggl(I-\frac{\mathcal A}{\lambda^{\alpha}}\Biggr)\int^{\infty}_{0}e^{-\lambda t}S(t)x\, dt,\quad \Re \lambda>0,\ x\in X
\end{align}
and
\begin{align}\label{olla}
\frac{Cy}{\lambda^{\beta+1}}=\int^{\infty}_{0}e^{-\lambda t}S(t)y\, dt -\frac{1}{\lambda^{\alpha}}\int^{\infty}_{0}e^{-\lambda t}S(t)x\, dt,
\end{align}
provided $\Re \lambda>0$ and $(x,y)\in X.$ Having in mind Lemma \ref{exp-c1c2c3-mlos-prim} and \eqref{lipo-preb}, it suffices to show that
\begin{align}
\notag
\frac{Cx}{\lambda^{1+\gamma}}& \in \lambda^{-(1+\gamma)} Cx-\lambda^{-(2+\beta +\gamma)}\int^{\infty}_{0}e^{-s/\lambda}S(s)x\, ds
\\\label{fak} & -\lambda^{-\alpha}{\mathcal A}^{-1}\Biggl[ \lambda^{-(1+\gamma)} Cx-\lambda^{-(2+\beta +\gamma)}\int^{\infty}_{0}e^{-s/\lambda}S(s)x\, ds \Biggr],\quad \Re \lambda>0,\  x\in X
\end{align}
and
\begin{align}
\notag
\frac{Cx}{\lambda^{1+\gamma}}&=\lambda^{-(1+\gamma)}Cx-\lambda^{-(2+\beta +\gamma)}\int^{\infty}_{0}e^{-s/\lambda}S(s)x\, ds
\\\label{fakh} & -\lambda^{-\alpha}\Biggl[\lambda^{-(1+\gamma)}Cx-\lambda^{-(2+\beta +\gamma)}\int^{\infty}_{0}e^{-s/\lambda}S(s)y\, ds\Biggr],
\end{align}
provided $\Re \lambda>0$ and $(y,x)\in {\mathcal A}.$ Keeping in mind the definition of operator ${\mathcal A}^{-1},$ the equation \eqref{fak} follows almost immediately from \eqref{ola}, while the equation \eqref{fakh} follows almost immediately from \eqref{olla}, with the number $\lambda$ replaced therein with the number $1/\lambda.$ 
\end{proof}

\begin{rem}\label{dometa}
\begin{itemize}
\item[(i)]
Keeping in mind the proof of \cite[Theorem 4.1(i)]{miao-li}, Theorem \ref{prostaku}
and 
Lemma \ref{exp-c1c2c3-mlos}(ii), the above result can be simply reformulated for the classes of exponentially equicontinuous
mild $(g_{\alpha},g_{\beta +1})$-regularized
$C_{1}$-existence families
and exponentially equicontinuous mild $(g_{\alpha},g_{\beta +1})$-regula-\\rized
$C_{2}$-uniqueness families. 
\item[(ii)]
Perturbation results for $(a,k)$-regularized $C$-resolvent families subgenerated by multivalued linear operators
have been investigated in \cite[Section 7]{FKP}. The interested reader may try to provide certain extensions of \cite[Corollary 4.2, Corollary 4.3]{miao-li} by using these results.
\item[(iii)] For applications, it will be crucial to reconsider and extend the conclusions obtained in 
\cite[Remark 4.2]{miao-li} for abstract degenerate fractional differential equations. Suppose that the operator family $\{(1+t^{\delta})^{-1}S(t) : t>0\}\subseteq L(X)$ is equicontinuous for some number $\delta\geq 0$ and all remaining assumptions in Theorem \ref{prostaku} hold.  
Then, for every non-negative real number $\gamma>2\delta +(1/2)-\beta,$ the operator ${\mathcal A}^{-1}$ is a subgenerator of an  $(g_{\alpha},g_{\gamma+1})$-regularized $C$-resolvent family $(R(t))_{t\geq 0}.$ Moreover, a simple calculation shows that the operator family $\{t^{-\gamma}(1+t^{\beta-\delta}+t^{\beta})^{-1}R(t) : t>0\}\subseteq L(X)$ is equicontinuous.
\end{itemize}
\end{rem}

It is crucial to formulate the following proper generalization of Theorem \ref{prostaku} (where $\omega_{0}'=0,$ $f(\lambda)=1/\lambda,$
$a(t)=b(t)=g_{\alpha}(t),$ $k(t)=g_{\beta+1}(t)
,$ $k_{1}(t)=g_{\gamma +1}(t),$ $g(t)=g_{\gamma-\beta}(t)$ and $S_{0}(t)=t^{1+\beta +\gamma}\int^{\infty}_{0}J_{1+\beta+\gamma}( 2\sqrt{st} ) s^{-\frac{1+\beta +\gamma}{2}}S(s)x\, ds,$ $t>0,\ x\in X;$
see also Remark \ref{dometa}) for various classes 
of $(a,k)$-regularized $C$-resolvent families:

\begin{thm}\label{pretis}
Suppose that ${\mathcal A}$ is a closed \emph{MLO} in $X$, $C,\ C_{2}\in L(X),$ $C_{1}\in L(Y,X),$ $C{\mathcal A} \subseteq {\mathcal A}C,$ $\omega_{0}'\geq \max(\emph{abs}(|a|),\emph{abs}(k),0),$ the functions $b(t)$ and $ k_{1}(t)$ satisfy \emph{(P1)} with $\omega_{0}\geq \max(0,\emph{abs}(|b|)),$ the function $k_{1}(t)$ is continuous for $t\geq 0$ and $|k_{1}(t)|=O(e^{\omega_{0}t}P(t))$ for $t\geq 0,$ where $P(t)=\sum_{j=0}^{l}a_{j}t^{\zeta_{j}},$ $t\geq 0$ ($l\in {\mathbb N},$ $a_{j}\geq 0$ and $\zeta_{j}\geq 0$ for $1\leq j\leq l$). Let
$f : \{ \lambda \in {\mathbb C} : \Re \lambda >\omega_{0}\} \rightarrow \{ \lambda \in {\mathbb C} : \Re \lambda >\omega_{0}'\}$ and $G: \{ \lambda \in {\mathbb C} : \Re \lambda >\omega_{0}\} \rightarrow {\mathbb C}$
be two given functions, let $\tilde{a}(\lambda)\neq 0$ for $\Re \lambda>\omega_{0}',$ and let 
\begin{align}\label{qwer}
\tilde{b}(\lambda)=\frac{1}{\tilde{a}(f(\lambda))}\ \ \mbox{ and }\ \ \widetilde{k_{1}}(\lambda)=G(\lambda) \tilde{k}(f(\lambda)),\quad \Re \lambda >\omega_{0}.
\end{align}
\begin{itemize}
\item[(i)] Suppose, further, ${\mathcal A}$ is a subgenerator of a global mild $(a,k)$-regularized $C_{1}$-existence family
$(R_{1}(t))_{t\geq 0}\subseteq L(Y,X)$ such that the operator family $\{e^{-\omega t}R_{1}(t) : t\geq 0\}\subseteq L(Y,X)$ is equicontinuous for each number $\omega>\omega_{0}'$,
as well as there exists a strongly continuous operator family $(S_{0}(t))_{t\geq 0}\subseteq L(Y,X)$ such that the family $\{e^{-\omega_{0} t}S_{0}(t) : t\geq 0\}\subseteq L(Y,X)$ is equicontinuous and
\begin{align}\label{brezob}
\int^{\infty}_{0}e^{-\lambda t}S_{0}(t)y\, dt=G(\lambda)\int^{\infty}_{0}e^{-sf(\lambda)}R_{1}(s)y\, ds,\quad y\in Y,\ \Re \lambda >\omega_{0}.
\end{align}
Then the operator ${\mathcal A}^{-1}$ is a subgenerator of a global mild $(b,k_{1})$-regularized $C_{1}$-existence family
$(S_{1}(t)\equiv k_{1}(t)C_{1}-S_{0}(t))_{t\geq 0}\subseteq L(Y,X)$ and the operator family $\{e^{-\omega_{0} t}(P(t))^{-1}S_{1}(t) : t> 0\}\subseteq L(Y,X)$ is equicontinuous.
\item[(ii)] Suppose, further, ${\mathcal A}$ is a subgenerator of a global mild $(a,k)$-regularized $C_{2}$-uniqueness family
$(R_{2}(t))_{t\geq 0}\subseteq L(X)$ such that the operator family $\{e^{-\omega t}R_{2}(t) : t\geq 0\}\subseteq L(X)$ is equicontinuous for each number $\omega>\omega_{0}'$,
as well as there exists a strongly continuous operator family $(S_{0}(t))_{t\geq 0}\subseteq L(X)$ such that the family $\{e^{-\omega_{0} t}S_{0}(t) : t\geq 0\}\subseteq L(X)$ is equicontinuous and \eqref{brezob} holds with $y=x\in X$ and $R_{1}(\cdot)$ replaced by $R_{2}(\cdot)$ therein.
Then the operator ${\mathcal A}^{-1}$ is a subgenerator of a global mild $(b,k_{1})$-regularized $C_{2}$-uniqueness family
$(S_{2}(t)\equiv k_{1}(t)C_{2}-S_{0}(t))_{t\geq 0}\subseteq L(X)$ and the operator family $\{e^{-\omega_{0} t}(P(t))^{-1} S_{2}(t) : t> 0\}\subseteq L(X)$ is equicontinuous.
\item[(iii)] Suppose, further, ${\mathcal A}$ is a subgenerator of a global $(a,k)$-regularized $C$-resolvent family
$(R(t))_{t\geq 0}\subseteq L(X)$ such that the operator family $\{e^{-\omega t}R(t) : t\geq 0\}\subseteq L(X)$ is equicontinuous for each number $\omega>\omega_{0}'$,
as well as there exists a strongly continuous operator family $(S_{0}(t))_{t\geq 0}\subseteq L(X)$ such that the family $\{e^{-\omega_{0} t}S_{0}(t) : t\geq 0\}\subseteq L(X)$ is equicontinuous and \eqref{brezob} holds with $y=x\in X$ and $R_{1}(\cdot)$ replaced by $R(\cdot)$ therein.
Then the operator ${\mathcal A}^{-1}$ is a subgenerator of a global mild $(b,k_{1})$-regularized $C$-resolvent family
$(S(t)\equiv k_{1}(t)C-S_{0}(t))_{t\geq 0}\subseteq L(X)$ and the operator family $\{e^{-\omega_{0} t}(P(t))^{-1}S(t) : t> 0\}\subseteq L(X)$ is equicontinuous.
\end{itemize}
\end{thm}

\begin{proof}
We will prove only (i). Let $\lambda \in {\mathbb C}$ with
$\Re\lambda>\omega_{0}$ and $\tilde{b}(\lambda)\widetilde{k_{1}}(\lambda) \neq 0$ be given.
Due to Proposition \ref{loza}(ii), it suffices to show that
$R(C_{1}) \subseteq
R(\tilde{b}(\lambda)-{\mathcal A})$
and
\begin{equation}\label{arenq-a-1}
\tilde{b}(\lambda)\widetilde{k_{1}}(\lambda) C_{1}y\in \bigl(\tilde{b}(\lambda)-{\mathcal A}\bigr)\int
\limits^{\infty}_{0}e^{-\lambda t}\bigl[k_{1}(t)C_{1}y-S_{1}(t)y\bigr]\,dt,\quad y\in Y.
\end{equation}
But, our assumption \eqref{qwer} implies $\tilde{a}(f(\lambda))\tilde{k}(f(\lambda)) \neq 0.$ Since $\Re (f(\lambda)) >\omega_{0}',$ Lemma \ref{exp-c1c2c3-mlos}(ii) yields that $R(C_{1}) \subseteq
R([\tilde{a}(f(\lambda))]^{-1}-{\mathcal A})$ and 
$$
\widetilde{R_{1}}(f(\lambda ))y\in  \frac{\tilde{k}(f(\lambda))}{\tilde{a}(f(\lambda))} \Bigl(\bigl[\tilde{a}\bigl(f(\lambda)\bigr)\bigr]^{-1}-{\mathcal A}\Bigr)^{-1}C_{1}y,\quad y\in Y.
$$ 
This simply gives \eqref{arenq-a-1} after a simple calculation involving the equations \eqref{qwer} and \eqref{brezob}.
\end{proof}

\begin{rem}\label{stvca}
In the existing literature concerning the inverse generator problem, the authors have investigated only the following special case: $\omega_{0}'=0,$ $f(\lambda)=1/\lambda$ and 
$a(t)=b(t)=g_{\alpha}(t),$ for some number $\alpha \in (0,2).$ Even if $f(\lambda)=1/\lambda ,$
the equality $a(t)=b(t)=g_{\alpha}(t),$ where $\alpha \in (0,2),$ is not necessary for applying Theorem \ref{pretis}. For example, suppose that $P(\lambda)=\sum_{j=0}^{n}a_{j}\lambda^{\zeta_{j}}$ and 
$Q(\lambda)=\sum_{j=0}^{m}b_{j}\lambda^{\eta_{j}}$ for some non-negative real numbers $\zeta_{j}$ ($0= \zeta_{0} \leq \zeta_{1} \leq \cdot \cdot \cdot \leq \zeta_{n}$), $\eta_{j}$ ($0= \eta_{0}\leq \eta_{1}\leq \cdot \cdot \cdot \leq \eta_{m}$)
and complex numbers $a_{j}$ ($0\leq j\leq n$), $b_{j}$ ($0\leq j\leq m$) such that $b_{0}=0,$ $a_{0}a_{n}b_{m}\neq 0,$ $\eta_{m}>\zeta_{n}$
and 
$P(\lambda)Q(\lambda)\neq 0$ for $\Re \lambda>0$ [we can take, for example, $P(\lambda)=\lambda+2$ and $Q(\lambda)=\lambda^{3}$]. If $a(t)={\mathcal L}^{-1}(P(\lambda)/Q(\lambda))(t),$ $t\geq 0,$ then we can prove that $\mbox{abs}(|a|)=0.$ Furthermore, we can prove that there exists a function $b(t)$ satisfying 
$\mbox{abs}(|b|)=0$ and $\tilde{b}(\lambda)=1/\tilde{a}(1/\lambda)=Q(1/\lambda)/P(1/\lambda),$ $\Re \lambda>0.$ So, if the second equality in \eqref{qwer} holds with the functions $k(t)$ and $k_{1}(t)$ being continuous for $t\geq 0$, then the most simplest case in which Theorem \ref{pretis}(iii) is applicable is that case in which $X:={\mathbb C},$ $C:=I,$
${\mathcal A}:=0$ and
$R(t):=k(t)I,$ $t\geq 0,$ when $S_{0}(t)=k_{1}(t)I,$ $t\geq 0.$
\end{rem}

\begin{rem}\label{brezob1}
In a great deal of concrete situations, it is almost impossible to represent $S_{0}(\cdot)$ in terms of $R_{1}(\cdot)$ directly, so that the use of complex characterization theorem for the Laplace transform is sometimes unavoidable.
\end{rem}

We continue by stating the following corollary of Theorem \ref{pretis} (the case in which $\sigma=-1$ and $a=b=\beta \geq 0$ has been already considered in Theorem \ref{prostaku} and Remark \ref{dometa}(i)-(ii)):

\begin{thm}\label{ras-sar}
Suppose that $\alpha \in (0,2),$ 
$\sigma \in (-1,0),$
$\beta \geq 0,$ ${\mathcal A}$ is a closed \emph{MLO} in $X$, $C\in L(X),$ $C{\mathcal A} \subseteq {\mathcal A}C,$ ${\mathcal A}$
is a subgenerator of an exponentially equicontinuous
$(g_{\alpha},g_{\beta +1})$-regularized
$C_{1}$-resolvent family $(R(t))_{t\geq 0}$ such that the operator family $\{(t^{a}+t^{b})^{-1}R(t) : t> 0\}\subseteq L(X)$ is equicontinuous for two real numbers $a,\ b$ such that $-1<a\leq b.$ Let $\eta>1+b$ and $\eta \geq 1+\beta.$
Define
$$
F(t):=t^{|\sigma|(\eta -\beta-1)}+\Bigl(t^{|\sigma|(\eta-b-1)}\chi_{(0,1]}(t)+t^{|\sigma|(\eta-a-1)}\chi_{[1,\infty)}(t)\Bigr),\quad t>0.
$$
Then the operator ${\mathcal A}^{-1}$ is a subgenerator of a global $(g_{\alpha |\sigma|},g_{1+|\sigma|(\eta -\beta-1)})$-regularized $C$-resolvent family
$(S(t))_{t\geq 0}\subseteq L(X)$ and the operator family $\{ [F(t)]^{-1} S(t) : t> 0\}\subseteq L(X)$ is equicontinuous.
\end{thm}

\begin{proof}
We will apply Theorem \ref{pretis}(iii) with $\omega_{0}=\omega_{0}'=0,$ $f(\lambda)=\lambda^{\sigma},$ $a(t)=g_{\alpha}(t),$ $b(t)=g_{|\sigma| \alpha}(t),$
$k(t)=g_{\beta +1}(t),$ $k_{1}(t)=g_{|\sigma|(\eta-\beta-1)\beta +1}(t)$ and $G(\lambda)=\lambda^{-1+\sigma \eta}.$
Set
$$
S_{0}(t)x:=\int^{\infty}_{0}t^{-\sigma \eta}\phi \bigl(-\sigma,1-\sigma \eta; -st^{-\sigma}\bigr)R(s)x\, ds,\quad t\geq 0,\ x\in X.
$$
Using \eqref{asimpt}, the assumption $-1<\sigma <0$ and the dominated convergence theorem, it readily follows that the mapping $t\mapsto S_{0}(t)x,$ $t> 0$ is continuous for every fixed element $x\in X.$
Since $S_{0}(0)=0$ and $[F(\cdot)-\cdot^{|\sigma|(\eta -\beta-1)}](0)=0,$ it suffices to show that the operator family $\{ [F(t)]^{-1} S(t) : t> 0\}\subseteq L(X)$ is equicontinuous as well as that
\begin{align*}
\int^{\infty}_{0}e^{-\lambda t}S_{0}(t)x\, dt=\lambda^{-1+\sigma \eta} \int^{\infty}_{0}e^{-sf(\lambda)}R(s)x\, ds,\quad x\in X,\ \Re \lambda >0.
\end{align*}
The asymptotic formula \eqref{asimpt} and the fact that the operator family $\{(t^{a}+t^{b})^{-1}R(t) : t> 0\}\subseteq L(X)$ is equicontinuous together imply  
that for each seminorm $p\in \circledast$ there exist a finite real constant $m>0$ and a seminorm $q\in \circledast$ such that
\begin{align*}
p\bigl(S_{0}(t)x\bigr) & \leq mq(x)t^{-\sigma \eta}\int^{\infty}_{0}\exp \Bigl(-m(st^{-\sigma})^{1/(1+\sigma)}\Bigr)\bigl( s^{a}+s^{b} \bigr)\, ds
\\ &= mq(x)t^{\sigma (1-\eta)}\int^{\infty}_{0}\exp \Bigl(-mr^{1/(1+\sigma)}\Bigr)\bigl(t^{\sigma a}r^{a}+t^{\sigma b}r^{b} \bigr)\, dr.
\end{align*}
Observe that the above integral converges due to our assumption $-1<a\leq b,$ which also implies $\sigma a \geq \sigma b$ and  the equicontinuity of operator family $\{ [F(t)]^{-1} S(t) : t> 0\}.$
Moreover, by the equation \eqref{lor} and the Fubini theorem, we have
\begin{align*}
\int^{\infty}_{0}& e^{-\lambda t}S_{0}(t)x\, dt
\\ &=\int^{\infty}_{0}R(s)x \cdot \int^{\infty}_{0}e^{-\lambda t} t^{-\sigma \eta}\phi \bigl(-\sigma,1-\sigma \eta; -st^{-\sigma}\bigr) \, dt \, ds
\\& =\lambda^{-1+\sigma \eta} \int^{\infty}_{0}e^{-sf(\lambda)}R(s)x\, ds,\quad x\in X,\ \Re \lambda >0.
\end{align*}
This completes the proof of theorem.
\end{proof}

As in Remark \ref{dometa}, it is worth noting that we can formulate the above result for the classes of exponentially equicontinuous
mild $(g_{\alpha},g_{\beta +1})$-regularized
$C_{1}$-existence families
and exponentially equicontinuous mild $(g_{\alpha},g_{\beta +1})$-regularized
$C_{2}$-uniqueness families (concerning the inverse generator problem, it is our duty to say that we have not been able to find certain applications with these classes of solution operator families). It is also worth noting the following:

\begin{rem}\label{suba}
Let the requirements of Theorem \ref{ras-sar} hold with $a=0$. Then the subordination principle for degenerate $(a,k)$-regularized $C$-resolvent families shows that the operator ${\mathcal A}$ is  a subgenerator of a global $(g_{\alpha |\sigma|},g_{1+|\sigma|\beta})$-regularized $C$-resolvent family $(W(t))_{t\geq 0}$ satisfying that the operator family $\{(1+t^{b|\sigma|})^{-1}W(t) : t\geq 0\}\subseteq L(X)$ is equicontinuous (\cite{FKP}). Arguing as in Remark \ref{dometa}, with the function $f(\lambda)=1/\lambda,$ we get that the operator ${\mathcal A}^{-1}$ is a subgenerator of a global $(g_{\alpha |\sigma|},g_{1+\gamma})$-regularized $C$-resolvent family $(W(t))_{t\geq 0}$ provided $\gamma \geq 0$ and $\gamma >2b|\sigma|+(1/2)-\beta |\sigma|.$ The integration rate obtained here with the function $f(\lambda)=\lambda^{\sigma}$ is better provided that
$|\sigma|(\eta-2b -1)<1/2.$
\end{rem}

Before we move ourselves to Subsection \ref{srbq}, we will provide one more example regarding the analysis of Poisson wave type equation in the space $L^{2}(\Omega),$ where $\emptyset \neq \Omega \subseteq {\mathbb R}^{n}$ is an open bounded domain with smooth boundary $\partial \Omega$ (see \cite[Example 2.3]{faviniyagi}):

\begin{example}\label{vgy}
Let $X:=H_{0}^{1}(\Omega) \times L^{2}(\Omega)$ and $m\in L^{\infty}(\Omega).$ Consider the bounded linear operator
$
M:=\begin{pmatrix}
1 & 0\\
0 & m(x)
\end{pmatrix},  
$
in $X,$ and an unbounded linear operator
$
L:=\begin{pmatrix}
0 & 1\\
\Delta & 0
\end{pmatrix},  
$
in $X,$ with domain $D(L):=[H^{2}(\Omega) \cap H_{0}^{1}(\Omega)] \times  H_{0}^{1}(\Omega).$ Then we know that the MLO ${\mathcal A}:=M^{-1}LM^{-1}-I$ satisfies $(0,\infty) \subseteq \rho({\mathcal A})$ and $\| R(\lambda : {\mathcal A}) \| \leq 1/\lambda,$ $\lambda>0.$ An application of \cite[Theorem 3.2.12]{FKP} gives that for each number $r>0$ the MLO ${\mathcal A}-I$ generates a global $(g_{1},g_{1+r})$-resolvent family $(S_{r}(t))_{t\geq 0}$ such that $\|S_{r}(t)\|=O(t^{r}),$ $t\geq 0.$ Suppose $\sigma \in (-1,0),$
$r>0$ and $\eta >1+r.$ Then, due to Theorem \ref{ras-sar}, the operator $({\mathcal A}-I)^{-1}$ is the integral generator of an  $(g_{|\sigma|},g_{1+|\sigma|(\eta -r-1)})$-resolvent family $(R_{r}(t))_{t\geq 0}$ such that $\|R_{r}(t)\|=O(t^{|\sigma|(\eta -r-1)}),$ $t\geq 0.$ Suppose that $(u_{0} \, v_{0})^{T} \in X,$   
$(u_{1}\,  v_{1})^{T} \in R({\mathcal A}-I),$ $(f_{1}(\cdot) \, f_{2}(\cdot))^{T}\in C([0,\infty) : X),$  
\begin{align*}
\bigl(g_{|\sigma|}\ast f_{1}\bigr)(t)+u_{0} =g_{1+|\sigma|(\eta -r-1)}(t)u_{1} \mbox{ and }\bigl(g_{|\sigma|}\ast f_{2}\bigr)(t)+v_{0} =g_{1+|\sigma|(\eta -r-1)}(t)v_{1},
\end{align*}
for any $t\geq 0.$ Since $D_{t}^{|\sigma|}(u(t)\, v(t))^{T}\in ({\mathcal A}-I)^{-1}(u(t)\, v(t))^{T}+(f_{1}(t)\, f_{2}(t))^{T},$ $t\geq 0$ is equivalent with 
$(u(t)\, v(t))^{T}\in ({\mathcal A}-I)[D_{t}^{|\sigma|}(u(t)\, v(t))^{T}-(f_{1}(t)\, f_{2}(t))^{T}],$ $t\geq 0,$ after a simple computation we get that the function $t\mapsto (u(t)\, v(t))^{T} \equiv R_{r}(t) (u_{1}\,  v_{1})^{T},$ $t\geq 0$ is a unique strong solution of the following system (see \cite[Definition 3.1.1]{FKP} for the notion): 
\begin{align*}
m(x)&\Bigl[ u(t,x)+D_{t}^{|\sigma|}u(t,x)-f_{1}(t,x) \Bigr]=D_{t}^{|\sigma|}v(t,x)-f_{2}(t,x),
\\ & m(x)\Bigl[ v(t,x)+D_{t}^{|\sigma|}v(t,x)-f_{2}(t,x) \Bigr]=\Delta \Bigl[D_{t}^{|\sigma|}u(t,x)-f_{1}(t,x)\Bigr]; 
\\ & u(0,x)=u_{0}(x),\ v(0,x)=v_{0}(x),\ x\in \Omega. 
\end{align*}
\end{example}

\subsection{Applications to degenerate time-fractional equations with abstract differential operators}\label{srbq}

Assume 
that
$n\in {\mathbb N}$ and $iA_{j},\ 1\leq j\leq n$ are commuting
generators of bounded $C_{0}$-groups on a Banach space $X.$\index{functional calculus for commuting
generators of bounded $C_{0}$-groups} Denote by
${\mathcal S}({{\mathbb R}^{n}})$ the Schwartz space\index{space!Schwartz} of
rapidly decreasing functions on ${{\mathbb R}^{n}}.$ Put $k:=1+\lfloor
n/2\rfloor,$ $A:=(A_{1},\cdot \cdot \cdot, A_{n})$ and
$A^{\eta}:=A_{1}^{\eta_{1}}\cdot \cdot \cdot A_{n}^{\eta_{n}}$ for
any $\eta=(\eta_{1},\cdot \cdot \cdot, \eta_{n})\in {{\mathbb
N}_{0}^{n}}.$
By ${\mathcal F}$ and ${\mathcal F}^{-1}$ we denote
the Fourier transform on ${\mathbb R}^{n}$ and its inverse transform, respectively.\index{Fourier transform}
Further on, for every $\xi=(\xi_{1},\cdot \cdot \cdot, \xi_{n}) \in
{{\mathbb R}^{n}}$ and $u\in {\mathcal F}L^{1}({\mathbb R}^{n})= \{
{\mathcal F}f : f \in L^{1}({{\mathbb R}^{n}}) \},$ we define
$|\xi|:=(\sum_{j=1}^{n}\xi_{j}^{2})^{1/2},$
$(\xi,A):=\sum_{j=1}^{n}\xi_{j}A_{j}$ and
$$
u(A)x:=\int_{{\mathbb R}^{n}}{\mathcal
F}^{-1}u(\xi)e^{-i(\xi,A)}x\, d\xi,\ x\in X.
$$
Then $u(A)\in
L(X),$ $u\in {\mathcal F}L^{1}({{\mathbb R}^{n}})$ and there exists a finite constant
$M\geq 1$ such that
$
\|u(A)\|\leq M \|{\mathcal
F}^{-1}u\|_{L^{1}({{\mathbb R}^{n}})},\ u\in {\mathcal
F}L^{1}({\mathbb R}^{n}).
$
Suppose $N\in {\mathbb N}$ and 
$P(x)=\sum_{|\eta|\leq N}a_{\eta}x^{\eta},$ $x\in {\mathbb R}^{n}$ is a complex polynomial. Then we define
$P(A):=\sum_{|\eta|\leq N}a_{\eta}A^{\eta}$ and
$X_{0}:=\{\phi(A)x : \phi \in {\mathcal S}({{\mathbb R}^{n}}),\ x\in
X\}.$
Let us recall that the operator $P(A)$ is
closable  (cf. \cite{FKP}-\cite{filomat} for further information).
Assuming that $X$ is a function space on which translations are uniformly
bounded and strongly continuous, then the usual option for $A_{j}$
is $-i\partial/\partial x_{j}$. If $P(x)=\sum_{|\eta|\leq N}a_{\eta}x^{\eta},$ $x\in
{{\mathbb R}^{n}}$ and $X$ is such a space (e.g.
$L^{p}({\mathbb R}^{n})$ with $p\in [1,\infty),$ $C_{0}({\mathbb
R}^{n})$ or $BUC({\mathbb R}^{n})$), then $\overline{P(A)}$ is
nothing else but the operator $\sum_{|\eta|\leq
N}a_{\eta}(-i)^{|\eta|}\partial^{|\eta|}/\partial
x_{1}^{\eta_{1}}\cdot \cdot \cdot \partial x_{n}^{\eta_{n}}\equiv
\sum_{|\eta|\leq N}a_{\eta}D^{\eta},$ acting with its maximal distributional
domain. Recall that $P(x)$ is called $r$-coercive ($0<r\leq
N$) iff there exist two finite real constants $M,\ L>0$ such that $|P(x)|\geq M|x|^{r},$
$|x|\geq L.$ 

Suppose now that $P_{1}(x)$ and $P_{2}(x)$ are non-zero complex polynomials in $n$ variables, 
and $0<\alpha < 2;$ set
$N_{1}:=dg(P_{1}(x)),$ $N_{2}:=dg(P_{2}(x))$ and
$m:=\lceil \alpha \rceil.$ In \cite[Section 4]{filomat}, 
we have analyzed the
generation of some specific classes of $(g_{\alpha},C)$-regularized resolvent families associated with
the following fractional degenerate abstract Cauchy problem\index{abstract Cauchy problem!fractional degenerate} \index{problem!(DFP)}
\[\hbox{(DFP)}: \left\{
\begin{array}{l}
{\mathbf D}_{t}^{\alpha}\overline{P_{2}(A)}u(t)=\overline{P_{1}(A)}u(t)+f(t),\quad t\geq 0,\\
u(0)=Cx;\quad u^{(j)}(0)=0,\ 1\leq j \leq \lceil \alpha \rceil -1,
\end{array}
\right.
\]
provided that
$P_{2}(x)\neq 0,$ $x\in {\mathbb R}^{n}$ and there exists a non-negative real number $\omega \geq 0$ such that
\begin{equation}\label{bound}
\sup_{x\in {\mathbb R}^{n}}\Re \Biggl(\Biggl(\frac{P_{1}(x)}{P_{2}(x)}\Biggr)^{1/\alpha}\Biggr)\leq
\omega,
\end{equation}
where $0^{1/\alpha}:=0.$ Our assumptions imply that the operator $\overline{P_{2}(A)}$ is injective but not invertible, in general (see e.g. \cite[Remark 8.3.5]{a43} and \cite[Remark 4.4(i)]{filomat}).

In the remainder of this subsection, we will focus our attention on the case $\omega=0.$ If $0\notin P_{1}( {\mathbb R}^{n})$ and $\omega=0,$ then we have
$
\sup_{x\in {\mathbb R}^{n}}\Re ((P_{2}(x)/P_{1}(x))^{1/\alpha})\leq
0,
$
so that the well-posedness of the reverse fractional degenerate abstract Cauchy problem\index{abstract Cauchy problem!fractional degenerate} \index{problem!(DFP)}
\[\hbox{(DFP)}_{r} : \left\{
\begin{array}{l}
{\mathbf D}_{t}^{\alpha}\overline{P_{1}(A)}u(t)=\overline{P_{2}(A)}u(t)+f(t),\quad t\geq 0,\\
u(0)=Cx;\quad u^{(j)}(0)=0,\ 1\leq j \leq \lceil \alpha \rceil -1
\end{array}
\right.
\]
can be analyzed as in \cite{filomat}. But, the real problems occur if $0\in P_{1}( {\mathbb R}^{n}),$ when the methods established in \cite{filomat} are inapplicable.
Our main result concerning the well-posedness of problem (DFP)$_{r}$ is stated as follows:

\begin{thm}\label{bubanj}
Suppose $0<\alpha<2,$ $\sigma \in (-1,0),$ $P_{1}(x)$ and $P_{2}(x)$ are non-zero complex polynomials,
$N_{1}=dg(P_{1}(x)),$ $N_{2}=dg(P_{2}(x)),$ $N\in
{\mathbb N}$ and $r\in (0,N].$ Let $Q(x)$ be an $r$-coercive complex
polynomial of degree $N,$ $a\in {\mathbb C} \setminus Q({\mathbb
R}^{n}),$ $\gamma>\frac{n\max (N,\frac{N_{1}+N_{2}}{\min(1,\alpha)})}{2r}$ (resp.
$\gamma=\frac{n}{r}|\frac{1}{p}-\frac{1}{2}|\max (N,\frac{N_{1}+N_{2}}{\min(1,\alpha)}),$ if
$E=L^{p}({\mathbb R}^{n})$ for some $1<p<\infty$), $P_{2}(x)\neq 0,$ $x\in {\mathbb R}^{n}$ and \eqref{bound} holds with $\omega=0.$
Set
$$
{\mathcal B}\equiv \overline{\overline{P_{2}(A)} \cdot \overline{P_{1}(A)}^{-1}},
$$
$
C:=((a-Q(x))^{-\gamma})(A),
$ $\delta:=\max(1,\alpha)n/2,$ if $E\neq L^{p}({\mathbb R}^{n})$ for all $p\in (1,\infty),$ and 
$\delta :=\max(1,\alpha)n|(1/p)-(1/2)|,$ if  $E= L^{p}({\mathbb R}^{n})$ for some $p\in (1,\infty).$
Then $C\in  L(X)$ is injective and the following holds:
\begin{itemize}
\item[(i)]
For each positive real number $\gamma>2\delta +(1/2)$ the multivalued linear operator ${\mathcal B}$ is a subgenerator of a global exponentially bounded $(g_{\alpha},g_{\gamma +1})$-regularized $C$-resolvent family $(R_{\alpha}(t))_{t\geq 0}$ satisfying that the operator family
$\{ [t^{\gamma}(1+t^{\beta-\delta}+t^{\beta})]^{-1}R_{\alpha}(t) : t>0 \} \subseteq L(X)$
is equicontinuous.
\item[(ii)] For each positive real number $\eta>1+\delta,$ the multivalued linear operator ${\mathcal B}$ is a subgenerator of a global exponentially bounded $(g_{\alpha |\sigma|},g_{|\sigma|(\eta -1)})$-regularized $C$-resolvent family $(R_{\alpha}(t))_{t\geq 0}$ satisfying that the operator family
$\{ [F(t)]^{-1}R_{\alpha}(t) : t>0 \} \subseteq L(X)$
is equicontinuous, where
$$
F(t)=t^{|\sigma|(\eta -1)}+\Bigl(t^{|\sigma|(\eta-\delta-1)}\chi_{(0,1]}(t)+t^{|\sigma|(\eta-1)}\chi_{[1,\infty)}(t)\Bigr),\quad t>0.
$$
\end{itemize}
\end{thm}

\begin{proof}
We will prove only (i).
Set 
$$
S_{\alpha}(t):=\Biggl(
E_{\alpha}\Biggl(t^{\alpha}\frac{P_{1}(x)}{P_{2}(x)}\Biggr)\bigl(a-Q(x)\bigr)^{-\gamma}\Biggr)(A),\
t\geq 0.
$$
Keeping in mind \cite[Theorem 4.2, Remark 4.4(i)]{filomat} and \cite[Proposition 2.1.2]{FKP}, we have that the operator $C\in  L(X)$ is injective as well as
$$
\lambda^{\alpha-1}\Bigl( \lambda^{\alpha}-\overline{P_{1}(A)} \cdot \overline{P_{2}(A)}^{-1}\Bigr)^{-1}Cx=\int^{\infty}_{0}e^{-\lambda t}S_{\alpha}(t)x\, dt,\quad \Re \lambda >0,\ x\in X. 
$$
By the equation \eqref{mare-care} and the first inclusion in \eqref{upozori}, the above implies that,
for every $\lambda \in {\mathbb C}$ with
$\Re\lambda>0,$ we have $R(C) \subseteq
R(I-\lambda^{-\alpha}{\mathcal B}^{-1}),$ (\ref{arendt1c1c2-mlosdf}) holds with $R_{1}(\cdot),$ $C_{1},$ ${\mathcal A}$ and $Y,$ $y$ replaced with $R(\cdot),$ $C,$
${\mathcal B}^{-1}$ and $X,$ $x$ therein, as well as
(\ref{arendt3-mlosdf}) holds with $R_{2}(\cdot),$ ${\mathcal A}$ and $C_{2}$ replaced with $S_{\alpha}(\cdot),$ ${\mathcal B}^{-1}$ and $C$ therein. Due to
Lemma \ref{exp-c1c2c3-mlos-prim}, we get that
$(S_{\alpha}(t))_{t\geq 0}\subseteq L(X)$ is a global exponentially bounded $(g_{\alpha},C)$-regularized resolvent family with a subgenerator ${\mathcal B}^{-1}$.
Moreover, we have
\begin{align}\notag
\bigl \|  S_{\alpha}(t)\bigr \|  & \leq
M\bigl(1+t^{\max(1,\alpha)n/2}\bigr),\ t\geq 0, \mbox{
resp., }
\\\label{procena1} &\bigl \|S_{\alpha}(t)\bigr \|\leq
M\bigl(1+t^{\max(1,\alpha)n|\frac{1}{p}-\frac{1}{2}|}\bigr),\ t\geq 0.
\end{align}
Using \eqref{procena1}, Theorem \ref{prostaku} and Remark \ref{dometa}(iii), the required assertion easily follows. 
\end{proof}

\begin{rem}\label{drugi-zajeb}
\begin{itemize}
\item[(i)]
We can choose the regularizing operator $C$ in a slightly different manner; see \cite[Theorem 4.3]{filomat}. For some refinements, see also \cite[Remark 4.4(ii)]{filomat}. 
\item[(ii)] The operator $\overline{P_{2}(A)} \cdot \overline{P_{1}(A)}^{-1}$ is closed provided that at least one of the operators $\overline{P_{1}(A)}$ or $\overline{P_{2}(A)}$ is invertible, when we have ${\mathcal B}=\overline{P_{2}(A)} \cdot \overline{P_{1}(A)}^{-1}.$
\item[(iii)] Consider the statement (i). The foregoing arguments also imply
\begin{align}\label{suljpa}
\lambda^{\alpha-\gamma-1}\overline{P_{1}(A)}\Bigl( \lambda^{\alpha}\overline{P_{1}(A)} -\overline{P_{2}(A)}\Bigr)^{-1}Cx=\int^{\infty}_{0}e^{-\lambda t}R_{\alpha}(t)x\, dt,\quad \Re \lambda >0,\ x\in X,
\end{align}
so that for each positive real number $\gamma>2\delta +(1/2)$ we have that $(R_{\alpha}(t))_{t\geq 0}$
is an exponentially bounded $(g_{\alpha},g_{\gamma +1})$-regularized $C$-resolvent family
for the abstract degenerate Cauchy problem
 \begin{equation}\label{a3}
 \overline{P_{1}(A)}u(t)=f(t)+\int \limits_{0}^{t}a(t-s) \overline{P_{2}(A)}u(s)\, ds,\ t\geq 0,
\end{equation}
in the sense of \cite[Definition 2.2]{filomat}. Moreover, conditions (i.2)-(i.3) and (ii.1) from the formulation of \cite[Theorem 2.8]{filomat} hold, 
so that for each $x\in D(\overline{P_{1}(A)}) \cap D(\overline{P_{2}(A)})$ the function $u(t) =R_{\alpha}(t)x,$ $t\geq 0$ is a unique strong solution of (\ref{a3}) with
$f(t)=g_{1+\gamma}(t)C\overline{P_{1}(A)}x,$ $t\geq 0$ (see \cite[Definition 2.1]{filomat} for the notion).
Moreover, we know that $(\overline{P_{2}(A)}^{-1}S_{\alpha}(t))_{t\geq 0}$ is an
exponentially equicontinuous $(g_{\alpha}, C)$-regularized resolvent family generated by $\overline{P_{2}(A)}, \ \overline{P_{1}(A)}$ in the sense of \cite[Definition 2.3.13]{FKP}.
Unfortunately, the operator $\overline{P_{1}(A)}$ need not be injective and the above fact cannot be used for the construction of $(g_{\alpha},g_{1+\gamma})$-regularized $C$-resolvent family generated by $\overline{P_{1}(A)}, \ \overline{P_{2}(A)}$ so that it is not clear how
we can consider the well-posedness of problem 
\[\hbox{(DFP)}_{L,r} : \left\{
\begin{array}{l}
\overline{P_{1}(A)}{\mathbf D}_{t}^{\alpha}u(t)=\overline{P_{2}(A)}u(t)+f(t),\quad t\geq 0,\\
u(0)=Cx;\quad u^{(j)}(0)=0,\ 1\leq j \leq \lceil \alpha \rceil -1,
\end{array}
\right.
\]
in general.
Similar conclusions can be formulated for the statement (ii).
\end{itemize}
\end{rem}

It is worth noting that Theorem \ref{bubanj} can be simply reformulated in $E_{l}$-type spaces (\cite{x263}-\cite{137}). In the following application of Theorem \ref{bubanj}, we consider the case in which $P_{2}(x)\equiv 1:$

\begin{example}\label{zvone}
\begin{itemize}
\item[(i)]
Suppose $0<\alpha<2.$
Let $E$ be one of the spaces $L^{p}({{\mathbb R}^{n}})$
($1\leq p\leq \infty),$ $C_{0}({{\mathbb R}^{n}}),$ $C_{b}({{\mathbb
R}^{n}}),$ $BUC({{\mathbb R}^{n}})$ and let $0\leq l\leq n.$
Let the operator ${\bf T_{l}}\langle \cdot \rangle ,$ the Fr\' echet space $X:=E_{l},$ the set ${{\mathbb
N_{0}^{l}}}$ and the seminorms $q_{\eta}(\cdot)$ (${{\mathbb
N_{0}^{l}}}$) possess the same meaning as
in \cite{137}, let $a_{\eta}\in {\mathbb C},\ 0\leq |\eta| \leq N$ and let
$P(D)f:=\sum_{|\eta| \leq N}a_{\eta}D^{\eta}f, $ with its maximal
distributional domain. Suppose that
$
\sup_{x\in {\mathbb R}^{n}}\Re (P(x)^{1/\alpha})\leq 0.
$
By Theorem \ref{bubanj}, we have that for each positive real number $\gamma>2\delta +(1/2),$ the multivalued linear operator $P(D)^{-1}$ generates an exponentially equicontinuous
$(g_{\alpha},g_{\gamma +1})$-regularized $C$-resolvent family $(R_{\alpha}(t))_{t\geq 0}.$ In our concrete situation, we have
\begin{align*}
 q_{\eta}\bigl(R_{\alpha}(t)f\bigr) \leq t^{\gamma}\bigl(1+t^{\beta-\delta}+t^{\beta}\bigr)q_{\eta}(f),\ t\geq 0,\ f\in E_{l},\ \eta \in {{\mathbb
N_{0}^{l}}}.
\end{align*}
The established conclusions can be applied in many different directions and here we will present only an application
of \cite[Proposition 3.2.15(ii)]{FKP}: Suppose $m\in {\mathbb N} \setminus \{1\},$ $x_{0}=x$, $f_{0}(\cdot)=f(\cdot),$ $P(D)x_{j}=x_{j-1}$ for $1\leq j\leq m,$
$P(D)f_{j}(t)=f_{j-1}(t)$ for $ t\geq 0$ and $1\leq j\leq m,$ $f_{j}\in C([0,\infty) : X)$ for $0\leq j\leq m,$
and $\alpha =1/m.$ Then the function
$v(t):=R_{\alpha}(t)x+(R_{\alpha}\ast C^{-1}f)(t)x,$ $t\geq 0$ is a unique
solution of the following abstract time-fractional equation:
\[\left\{
\begin{array}{l}
v\in C^1((0,\infty):X)\cap C([0,\infty):X),\\
v(t)= P(D)\Biggl[v_{t}(t,x)-\sum \limits_{j=1}^{m-1}g_{(j/m)+r}(t)Cx_{j}\\ \ \ \ \
\ \ \  -\sum \limits_{j=0}^{m-1}\bigl( g_{(j/m)+\gamma}\ast f_{j}
\bigr)(t)-
g_{\gamma}(t)Cx\Bigr],\ t>0,\\
v(0)=0.
\end{array}
\right.
\]
Furthermore, $v\in C^{1}([0,\infty):X)$ provided that $\gamma\geq 1$ or
$x=0.$ 
\item[(ii)] It is worth noting that the operator $P(D)$ need not be injective. For example, if $l=0,$ $n=1$ and $P(x)=-x^{2}+ix,$ $x\in {\mathbb R},$ then the state space is $X=BUC({\mathbb R})$ but the operator $P(D)\cdot=\cdot^{\prime \prime}+\cdot^{\prime}$ is not injective since $P(D)f=0$ for all constant functions $f(\cdot);$ in particular, the operator $P(D)^{-1}$ is not single-valued and $P(D)^{-1}\notin L(X).$ On the other hand, the injectiveness of operator $P(D)$ holds in many concrete situations; for example, the Korteweg-De Vries operator $P(D)\cdot=\cdot^{\prime \prime \prime}+\cdot^{\prime}$ in $L^{p}({\mathbb R})$ is injective ($1\leq p<\infty$), the Laplacian $\Delta$ in  $L^{p}({\mathbb R}^{n})$ is injective and the assumption $0\notin P({\mathbb R}^{n})$ implies that the operator $P(D)$ is injective, as well (see \cite[Remark 2.2.22(i), Lemma 2.3.23]{FKP}).     
\end{itemize}
\end{example}

We close the paper with the following adaptation of \cite[Example 4.9(ii)]{filomat}, where the author has obtained some wrong conclusions by disregarding condition 
$0\notin P_{2}({\mathbb R}^{n}):$

\begin{example}\label{2.111qzidio}
Suppose that $P_{1}(x)=-|x|^{2}$ and $P_{2}(x)=\sum_{|\eta|\leq Q}a_{\eta}x^{\eta}$ ($x\in {\mathbb R}^{n}$), $0\notin P_{2}({\mathbb R}^{n})$ and \eqref{bound} holds with $\omega=0.$ By Theorem \ref{bubanj}, there exist a non-negative real number $\gamma \geq 0$ and an injective operator $C\in L(X)$ such that the operator ${\mathcal B}$
is a subgenerator of a global polynomially bounded $(g_{\alpha},g_{1+\gamma})$-regularized $C$-resolvent family $(R_{\alpha}(t))_{t\geq 0}$ (cf. also the conclusions from Remark \ref{drugi-zajeb}(ii)-(iii)),
so that we can analyze the existence and uniqueness of strong (mild) solutions of the following fractional degenerate Cauchy problem
of order $\alpha \in (0,2):$
\[
(P)_{\alpha} : \left\{
\begin{array}{l}
{\mathbf D}_{t}^{\alpha}u_{xx}(t,x)
=\sum_{|\eta|\leq Q}a_{\eta}D^{\eta}u(t,x)+f(t,x),\quad t\geq 0,\ x\in {\mathbb R}^{n},\\
u(0,x)=C\phi(x); \ u_{t}(0,x)=C\psi(x)\mbox{ if }\alpha>1.
\end{array}
\right.
\]
Without going into full details, we will only note that in the case $\alpha \in (1,2),$ the validity of certain conditions on
the initial values $\phi(x),\ \psi(x)$ and the inhomogenity $f(t,x)$  yield that the unique strong solution of $(P)_{\alpha}$ is given by
$$
u(t,x):={\mathbf D}_{t}^{1+\gamma}\Biggl[ R_{\alpha}(t)\phi(x)+\int^{t}_{0}R_{\alpha}(s)\psi(x)\, ds+\bigl( g_{\alpha-1} \ast R_{\alpha} \ast f  \bigr)(t,x) \Biggr],\ t\geq 0,\ x\in {\mathbb R}^{n}.
$$ 
For more details, see \cite[Theorem 13]{yy-li0} and \cite[Lemma 4.1]{yy-li}.
\end{example}

\end{document}